\documentclass[a4paper,12pt]{amsart}
\usepackage{amssymb}
\usepackage{ifthen}
\usepackage[usenames]{color}
\usepackage{graphicx}
\nonstopmode \numberwithin{equation}{section}
\newtheorem{thm}{Theorem}[section]

\newtheorem{cor}{Corollary}[section]
\newtheorem{lem}{Lemma}[section]
\newtheorem{prop}{Proposition}[section]

\newtheorem{claim}{Claim}[section]
\newtheorem{subclaim}{Subclaim}
\newtheorem{conj}[equation]{Conjecture}
\newtheorem{case}{Case}[section]
\newtheorem*{mysolution}{Solution}
\newtheorem{step}{Step}[section]
\theoremstyle{definition}
\newtheorem{defn}{Definition}[section]
\newtheorem{examp}{Example}[section]
\newtheorem{prob}[equation]{Problem}
\newtheorem{ques}[equation]{Question}
\newtheorem{rem}{Remark}[section]
\newcounter {own}
\def\theown {\thesection       .\arabic{own}}

\newenvironment{pf}[1][]{%
 \vskip 3mm
 \noindent
 \ifthenelse{\equal{#1}{}}%
  {{\slshape Proof. }}%
  {{\slshape #1.} }%
 }%
{\qed\bigskip}

\newcounter{alphabet}
\newcounter{tmp}
\newenvironment{Thm}[1][]{\refstepcounter{alphabet}%
\bigskip%
\noindent%
{\bf Theorem \Alph{alphabet}}%
\ifthenelse{\equal{#1}{}}{}{ (#1)}%
{\bf .} \itshape}{\vskip 8pt}

\makeatletter

\newcommand{\Rref}[1]{\@ifundefined{r@#1}{}{\setcounter{tmp}{\ref{#1}}\Alph{tmp}}}

\def\be{\begin{equation}}
\def\ee{\end{equation}}

\newcommand{\ben}{\begin{enumerate}}
\newcommand{\een}{\end{enumerate}}

\newcommand{\blem}{\begin{lem}}
\newcommand{\elem}{\end{lem}}
\newcommand{\bthm}{\begin{thm}}
\newcommand{\ethm}{\end{thm}}
\newcommand{\bcor}{\begin{cor}}
\newcommand{\ecor}{\end{cor}}
\newcommand{\beg}{\begin{examp}}
\newcommand{\eeg}{\end{examp}}
\newcommand{\begs}{\begin{examples}}
\newcommand{\eegs}{\end{examples}}
\newcommand{\bdefe}{\begin{defn}}
\newcommand{\edefe}{\end{defn}}
\newcommand{\bprob}{\begin{prob}}
\newcommand{\eprob}{\end{prob}}
\newcommand{\bques}{\begin{ques}}
\newcommand{\eques}{\end{ques}}
\newcommand{\bei}{\begin{itemize}}
\newcommand{\eei}{\end{itemize}}
\newcommand{\bcl}{\begin{claim}}
\newcommand{\ecl}{\end{claim}}
\newcommand{\bscl}{\begin{subclaim}}
\newcommand{\escl}{\end{subclaim}}
\newcommand{\bca}{\begin{case}}
\newcommand{\eca}{\end{case}}
\newcommand{\bstep}{\begin{step}}
\newcommand{\estep}{\end{step}}
\newcommand{\bsol}{\begin{mysolution}}
\newcommand{\esol}{\end{mysolution}}
\newcommand{\bcon}{\begin{conj}}
\newcommand{\econ}{\end{conj}}
\newcommand{\bcons}{\begin{conjs}}
\newcommand{\econs}{\end{conjs}}
\newcommand{\bprop}{\begin{prop}}
\newcommand{\eprop}{\end{prop}}
\newcommand{\br}{\begin{rem}}
\newcommand{\er}{\end{rem}}
\newcommand{\brs}{\begin{rems}}
\newcommand{\ers}{\end{rems}}
\newcommand{\bo}{\begin{obser}}
\newcommand{\eo}{\end{obser}}
\newcommand{\bos}{\begin{obsers}}
\newcommand{\eos}{\end{obsers}}
\newcommand{\bpf}{\begin{pf}}
\newcommand{\epf}{\end{pf}}
\newcommand{\ba}{\begin{array}}
\newcommand{\ea}{\end{array}}
\newcommand{\beq}{\begin{eqnarray}}
\newcommand{\beqq}{\begin{eqnarray*}}
\newcommand{\eeq}{\end{eqnarray}}
\newcommand{\eeqq}{\end{eqnarray*}}

\begin{document}

\title{Schwarz lemma for the solutions to the Dirichlet problems for the invariant Laplacians}

\author{Qianyun Li}
\address{Qianyun Li, Key Laboratory of Computing and Stochastic Mathematics (Ministry of Education), School of Mathematics and Statistics, Hunan Normal University, Changsha, Hunan 410081, P. R. China}
\email{liqianyun@hunnu.edu.cn}

\author{Jiaolong Chen${}^{~\mathbf{*}}$}
\address{Jiaolong Chen, Key Laboratory of Computing and Stochastic Mathematics (Ministry of Education), School of Mathematics and Statistics, Hunan Normal University, Changsha, Hunan 410081, P. R. China}
\email{jiaolongchen@sina.com}

\keywords{Invariant Laplacians, Schwarz Lemma, Hardy space.\\
$^{\mathbf{*}}$Corresponding author}

\subjclass[2010]{Primary 31B05; Secondary 42B30.}

\maketitle

\makeatletter\def\thefootnote{\@arabic\c@footnote}\makeatother

\begin{abstract}
The main purpose of this paper is to establish a Schwarz lemma for
 the solutions to the Dirichlet problems for the invariant Laplacians.
The obtained
result  of this paper is a generalization  of the corresponding known
results  (\cite[Theorem 1.1]{Chen21}  and \cite[Theorem 2.1]{kalaj2018}).
\end{abstract}

\maketitle

\section{Introduction}\label{sec-1}
For $n\ge 2$, let $\mathbb{R}^{n}$ denote the usual real vector space of dimension $n$.
For $x=(x_{1},\ldots,x_{n})\in \mathbb{R}^{n}$, sometimes we  identify each point $x$ with a column vector.
For two column vectors $x,y\in \mathbb{R}^{n}$, we use $\langle x,y\rangle$ to denote the inner product of $x$ and $y$.
 We denote by $\mathbb{B}^n$  the unit ball in $\mathbb{R}^n$ and $\mathbb{B}_n$   the unit ball in $\mathbb{C}^n\cong\mathbb{R}^{2n}$.
For $A=\big(a_{ij}\big)_{n\times n}\in \mathbb{R}^{n\times n}$,
 the matrix norm of $A$ is defined by
 $\|A\|=\sup\{|A\xi|:\; \xi\in \mathbb{S}^{n-1}\}$,
 where $\mathbb{S}^{n-1}=\{x\in \mathbb{R}^{ n}:|x|=1\}$.

In this paper, we consider  the following family of differential operators
$$
	\Delta_{\alpha} =\left(1-|x|^{2}\right)\left(\frac{1-|x|^{2}}{4} \sum_{i=1}^{n} \frac{\partial^{2}  }{\partial x_{i}^{2}} +\alpha\sum_{i=1}^{n} x_{i} \frac{\partial  }{\partial x_{i}}+  \alpha\left(\frac{n}{2}-1-\alpha\right)   \right)
$$
 for $\alpha\in \mathbb{R}$  (cf.  \cite{Liu04, Liu09}).

In this paper, we call $\Delta_{\alpha}$ the (M\"{o}bius) invariant Laplacian, since
\begin{equation}\label{eq-1.1}
 \Delta_{\alpha}\left\{\left(\operatorname{det} D\psi (x)\right)^{\frac{n-2-2 \alpha}{2 n}} u(\psi(x))\right\}=\left(\operatorname{det} D\psi (x)\right)^{\frac{n-2-2 \alpha}{2 n}}\left(\Delta_{\alpha} u\right)(\psi(x))
\end{equation}
for every  $u \in C^{2}\left(\mathbb{B}^n,\mathbb{R}^n\right)$ and for every  $\psi \in M (\mathbb{B}^n )$ (see \cite[Proposition 3.2]{Liu09}).
Here  $ M (\mathbb{B}^n )$  denotes the group of those M\"{o}bius transformations that map   $\mathbb{B}^{n}$ onto  $\mathbb{B}^{n}$,
  and $D \psi (x)$  denotes the Jacobian matrix of $\psi$.

If $u \in C^{2}\left(\mathbb{B}^n,\mathbb{R}^n\right)$ and $\Delta_{\alpha}u =0$ in $ \mathbb{B}^{n}$,
 then for any unitary transformation $A$ in $\mathbb{R}^{n}$,
 \eqref{eq-1.1} guarantees that
 $$\Delta_{\alpha}(u\circ A)  = 0.$$

 Obviously,  $\Delta_{\alpha}$   is the Laplace-Beltrami operator  $\Delta_{h}$  when  $\alpha=n / 2-1$ (cf.  \cite{chen2018, Chen21}).
The differential operator  $\Delta_{\alpha}$ can also be regarded as the real counterpart  of the $(\alpha,\beta)$-Laplacian $\Delta_{\alpha,\beta}$
in $\mathbb{B}_{n}$ given by
$$
	\Delta_{\alpha,\beta} =4\left(1-|z|^{2}\right)\left(\sum_{i,j}(\delta_{i,j}-z_{i}\overline{z_{j}}) \frac{\partial^{2} }{\partial z_{i}\partial \overline{z_{j}}}+\alpha\sum_{j}  z_{j} \frac{\partial  }{\partial z_{j}}+\beta\sum_{j} \overline{ z_{j}} \frac{\partial  }{\partial \overline{z_{j}}}-\alpha\beta \right),
$$
which was introduced by Geller \cite{gell} and extensively studied by many authors (see \cite{A1, A2, A3, B4, Z5}).
On the other hand, the operator  $\Delta_{\alpha}$ is closely related to polyharmonic mappings.  Indeed, the differential operator  $\mathbf{L}_{\alpha}:=   \left(4 /\left(1-|x|^{2}\right)\right) \Delta_{\alpha}$  plays a crucial role in the modified Almansi representation for polyharmonic mappings in \cite{Liu21} as well as in the cellular decomposition theorem for polyharmonic mappings in \cite{AH2014}.

In \cite{Liu04}, Liu and Peng  considered the  existence and the representation of the solutions to the following
Dirichlet problem:
\begin{equation}\label{eq-1.2}
\left\{\begin{array}{ll}
	\Delta_{\alpha} u(x)=0, & x\in\mathbb{B}^{n}, \\
	u(\zeta)=f(\zeta),& \zeta\in\mathbb{S}^{n-1}.
\end{array}\right.
\end{equation}

It is known that if   $u\in C^{2}(\mathbb{B}^{n},\mathbb{ R}^{n})\cap C (\overline{\mathbb{B}}^{n},\mathbb{ R}^{n})$, then the Dirichlet problem \eqref{eq-1.2}
has a solution  if and only if $\alpha>-1 / 2$,
 where $\overline{\mathbb{B}}^{n}=\mathbb{B} ^{n}\cup \mathbb{S}^{n-1}$ (see \cite[Theorem 2.4]{Liu04}).
 In this case the solution is unique and   can be expressed by the following Poisson type integral:
\begin{equation}\label{eq-1.3}
u(x)=\int_{\mathbb{S}^{n-1}} P_{\alpha}(x, \zeta) f(\zeta) d \sigma(\zeta)=P_{\alpha}[f](x),	
\end{equation}
where $d \sigma$  denotes the normalized surface measure on  $\mathbb{S}^{n-1}$   so that $\sigma(\mathbb{S}^{n-1})=1$,
and
 \begin{equation}\label{eq-1.4}
P_{\alpha}(x, \zeta)=C_{n, \alpha} \frac{\left(1-|x|^{2}\right)^{1+2 \alpha}}{|x-\zeta|^{n+2 \alpha}}
\;\;\text{with}\;\;C_{n, \alpha}=\frac{\Gamma\left(\frac{n}{2}+\alpha\right) \Gamma(1+\alpha)}{\Gamma\left(\frac{n}{2}\right) \Gamma(1+2 \alpha)}.
\end{equation}

Further, if $u$ is the solution to \eqref{eq-1.2} and  $f=\sum_{k} Y_{k}$  is the spherical harmonic expansion of  $f$,
then for any $x=r\zeta\in \mathbb{B}^{n}$, we have
$$	u(x)=\sum_{k=0}^{\infty} \frac{\Phi_{k}^{\alpha}\left(r^{2}\right)}{\Phi_{k}^{\alpha}(1)}r^{k} Y_{k} (\zeta),
$$	
where $Y_{k} \in \mathcal{H}_{k}(\mathbb{S}^{n-1},\mathbb{ R}^{n})$  and
 \begin{equation}\label{eq-1.5}
\Phi_{k}^{\alpha} (r^{2} )={_{2}F_{1}}\left(-\alpha,k+\frac{n}{2}-1-\alpha,k+\frac{n}{2};r^{2} \right)
 \end{equation}
(see \cite[Theorem 2.4]{Liu04}).
Here and hereafter,	 $\mathcal{H}_{k}\left(\mathbb{S}^{n-1},\mathbb{R}^{n}\right)$ denotes the space of
	spherical harmonic mappings of degree $k$ from $\mathbb{S}^{n-1}$ into $\mathbb{ R}^{n}$.

For the case when $n=2$, we refer  to  \cite{Adel2021,  AH2014, Chen23, Long2022, ol2014} for basic properties
of the mappings defined in  \eqref{eq-1.3}.
 For higher dimensional cases, see \cite{Liu04,Liu09,Zhou22}.
 In this paper, we focus our investigations on the case when $n\geq3$.

For $p\in(0,\infty]$, we use  $L^{p}(\mathbb{S}^{n-1},\mathbb{ R}^{n})$ to denote the space of   those  measurable  mappings
$f: \mathbb{S}^{n-1}\rightarrow\mathbb{ R}^{n}$ such that  $\|f\|_{L^p}<\infty$, where
$$\|f\|_{L^{p}}=\begin{cases}
\displaystyle \;\left(\int_{\mathbb{S}^{n-1}}|f(\zeta)|^{p} d \sigma(\zeta)\right)^{1 / p} , & \text{ if } p\in (0,\infty),\\
\displaystyle \;\text{esssup}_{\zeta\in \mathbb{S}^{n-1}}  |f( \zeta)|  , \;\;\;\;& \text{ if } p=\infty.
\end{cases}$$

For a finite signed  Borel measure $\mu$ on $\mathbb{S}^{n-1}$ and $x\in \mathbb{B}^{n}$,
we define
\begin{equation}\label{eq-1.6}
P_{\alpha}[\mu](x)=\int_{\mathbb{S}^{n-1}} P_{\alpha}(x, \zeta) d \mu(\zeta) .	
\end{equation}
Differentiating under the integral sign in (\ref{eq-1.6}), we see that $\Delta_{\alpha}(P_{\alpha}[\mu])=0$ in $\mathbb{B}^{n}$.

It is natural to identify each $f \in L^{1}(\mathbb{S}^{n-1}, \mathbb{R})$ with the measure $\mu_{f}  $ defined on Borel sets $E \subset \mathbb{S}^{n-1} $ by
\begin{equation*}
	\mu_{f}(E)=\int_{E} f d \sigma.
\end{equation*}
Therefore,  $d \mu_{f}=f d \sigma$.

 For convenience,  throughout this paper, we use $P_{\alpha}[f]$ to denote the   Poisson type integral of $f\in L^{1}(\mathbb{S}^{n-1},\mathbb{R}^{n})$ in $\mathbb{B}^{n}$, where
$$
P_{\alpha}[f](x)=\int_{\mathbb{S}^{n-1}} P_{\alpha}(x, \zeta) f(\zeta) d \sigma(\zeta).	
$$


For $p\in(0,\infty]$, the {\it Hardy space} $\mathcal{H}^{p}(\mathbb{B}^{n}, \mathbb{ R}^{n} )$ consists of all those mappings
$f: \mathbb{B}^{n}\rightarrow\mathbb{ R}^{n}$ such that $f$ is measurable, $M_{p}(r,f)$ exists for all $r\in (0,1)$ and $\|f\|_{\mathcal{H}^p}<\infty$, where
$$\|f\|_{\mathcal{H}^p}=\sup_{0<r<1} M_{p}(r,f) $$
and
$$\;\;M_{p}(r,f)=\begin{cases}
\displaystyle \;\left( \int_{\mathbb{S}^{n-1}} |f(r\zeta)|^{p}d\sigma(\zeta)\right)^{\frac{1}{p}} , & \text{ if } p\in (0,\infty),\\
\displaystyle \;\sup_{\zeta\in \mathbb{S}^{n-1}}  |f(r\zeta)|  , \;\;\;\;& \text{ if } p=\infty.
\end{cases}$$

Similarly, for $p\in(0,\infty]$, we use $H^p(\mathbb{B}_{n},\mathbb{C}^{n})$ to denote
the Hardy space of mappings from $\mathbb{B}_n$ into $\mathbb{C}^{n}$.

The classical Schwarz lemma states that $|f(z)|\leq |z|$ for every holomorphic mapping of the unit disk $\mathbb{B}_1$
 into itself satisfying the condition $f(0)=0$.
In 1950,  Macintyre and Rogosinski generalized the above result to the setting of holomorphic mappings in  $H^p(\mathbb{B}_1,\mathbb{C})$ (see \cite[Section 14]{maro}).
They proved that
if $f$ is a holomorphic mapping in $\mathbb{B}_1$ such that $f(0)=0$  and $\|f\|_{H^p}<\infty$ with  $p\in[1,\infty)$,
then the following inequality is sharp:
$$
|f(z)|\le \frac{|z|}{(1-|z|^2)^{1/p}}\|f\|_{H^p},
$$
where $z\in \mathbb{B}_1$.
For the high dimensional case, see \cite[Theorem~4.17]{zhu}.

The classical Schwarz lemma for harmonic mappings (cf. \cite[Lemma 6.24]{ABR}) states that
if $f\in \mathcal{H}^{\infty}(\mathbb{B}^{n},\mathbb{R}^{n})$ is a  harmonic mapping with $f(0)=0$, then
$$
|f(x)|\le U(|x|e_{n})\|f\|_{\mathcal{H}^{\infty}}.
$$
Here $e_{n}=(0,\dots,0,1)\in \mathbb{S}^{n-1}$ and $U$ is a harmonic function of $\mathbb{B}^{n}$ into $[-1,1]$ defined by
$$
U(x)= P[\chi_{\mathbb{S}^{n-1}_{+}}-\chi_{\mathbb{S}^{n-1}_{-}}](x),
$$
where $\chi$ is the indicator function,
$\mathbb{S}^{n-1}_{+}=\{x\in \mathbb{S}^{n-1}: x_n \geq 0\},$ $\mathbb{S}^{n-1}_{-}=\{x\in \mathbb{S}^{n-1}: x_n \leq 0\}$ and
$P[\chi_{\mathbb{S}^{n-1}_{+}}-\chi_{\mathbb{S}^{n-1}_{-}}]$ is the Poisson integral of $\chi_{\mathbb{S}^{n-1}_{+}}-\chi_{\mathbb{S}^{n-1}_{-}}$ with respective to $\Delta$.
Recently, Kalaj \cite[Theorem 2.1]{kalaj2018} considered harmonic mappings $f$ in
 $\mathcal{H}^p(\mathbb{B}^{n},\mathbb{R}^{n})$ with $p\in[1,\infty]$ and $f(0)=0$.
For those mappings, he obtained the following two sharp estimates:
$$
|f(x)|\le g_p(|x|)\|f\|_{\mathcal{H}^p}\;\;\text{and}\;\;
\|Df(0)\|\le n \left(\frac{\Gamma\left(\frac{n}{2}\right) \Gamma\left(\frac{1+q}{2}\right)}{\sqrt \pi\Gamma\left(\frac{n+q}{2}\right)}\right)^{\frac{1}{q}}\|f\|_{\mathcal{H}^p},
$$
where $x\in \mathbb{B}^{n}$, $g_{p}$ is a smooth mapping and $q$ is the conjugate of $p$.
See \cite{bur, Chen21} for related discussions on harmonic mappings and hyperbolic harmonic mappings.

The main purpose of this paper is to consider the Schwarz lemma for the solutions to the Dirichlet problems \eqref{eq-1.2}.
 Before the statement of the
result, for convenience, we introduce the following notational conventions.

 Let $r\in[0,1)$, $p\in[1,\infty]$ and $q$ be its conjugate. Define
\begin{equation}\label{eq-1.7}
G_p(r)=\left\{
            \begin{array}{ll}
              \inf_{a\in [0,\infty)}\sup_{\eta\in \mathbb{S}^{n-1}} |P_\alpha(re_{n},\eta)-a|, & \hbox{if $q=\infty$;} \\
              \inf_{a\in [0,\infty)}\left(\int_{\mathbb{S}^{n-1}} |P_\alpha(re_{n},\eta)-a|^qd\sigma(\eta)\right)^{1/q}, & \hbox{if $q\in[1,\infty)$}.
            \end{array}
          \right.
\end{equation}
Then we have the following result.
\begin{thm}\label{thm-1.1}
Suppose that $u=P_{\alpha}[f]$ and $u(0)=0$, where
$f\in L^p(\mathbb{S}^{n-1},\mathbb{R}^{n})$ with $p\in[1,\infty]$  and $\alpha\in(-\frac{1}{2},+\infty)$.
 Then for any $x\in \mathbb{B}^n$,
\begin{equation}\label{eq-1.8}
|u(x)|\le G_p(|x|)\|f\|_{L^{p}},
\end{equation}
 and
 \begin{equation}\label{eq-1.9}
	\|D u(0)\| \leq C_{n,\alpha}(n+2\alpha)\left(\frac{\Gamma\left(\frac{n}{2}\right) \Gamma\left(\frac{1+q}{2}\right)}{\sqrt{\pi} \Gamma\left(\frac{n+q}{2}\right)}\right)^{\frac{1}{q}}\|f\|_{L^{p}} .
 \end{equation}
Both inequalities \eqref{eq-1.8} and \eqref{eq-1.9} are sharp.

In particular, if $p\in[1,\infty )$ and $\alpha\in(-\frac{1}{2},0]\cup[\frac{n}{2}-1,+\infty)$,
then $G_p(r)$  is an increasing homeomorphism  of $[0,1)$ into $[0,\infty)$ with $G_p(0)=0$ and $ \frac{d}{dr}G_p(r)$ is continuous in $[0,1)$;
if $p=\infty$ and $\alpha\in(-\frac{1}{2},0]\cup[\frac{n}{2}-1,+\infty)$,
then $G_{\infty}(r)=U_{\alpha}(re_{n})$
 is an increasing homeomorphism of $[0,1)$ onto itself and $ \frac{d}{dr}G_{\infty}(r)$ is continuous in $[0,1)$,
where
$U_{\alpha}=P_{\alpha}[\chi_{\mathbb{S}_{+}^{n-1}} -\chi_{\mathbb{S}_{-}^{n-1}}]$.
\end{thm}

\begin{rem}
\begin{enumerate}
\item It seems unlikely that we can explicitly express the function $G_p(r)$ for general $p$. However we demonstrate some special cases $p=1,2,\infty$ in Section~\ref{sec-4}.
\item Theorem \ref{thm-1.1} is a generalization of
\cite[Theorem 1.1]{Chen21}  and \cite[Theorem 2.1]{kalaj2018}.
\item If $\alpha\in(0,\frac{n}{2}-1)$, then $G_{p}(r)$ doesn't always increase in $[0,1)$.
This can be seen by considering the case when $p=\infty$, $\alpha=1$ and $n=6$ (see Figure \ref{fig1}).
 \begin{figure}
 	\centering
 	\includegraphics[width=0.5\columnwidth,height=0.3\linewidth]{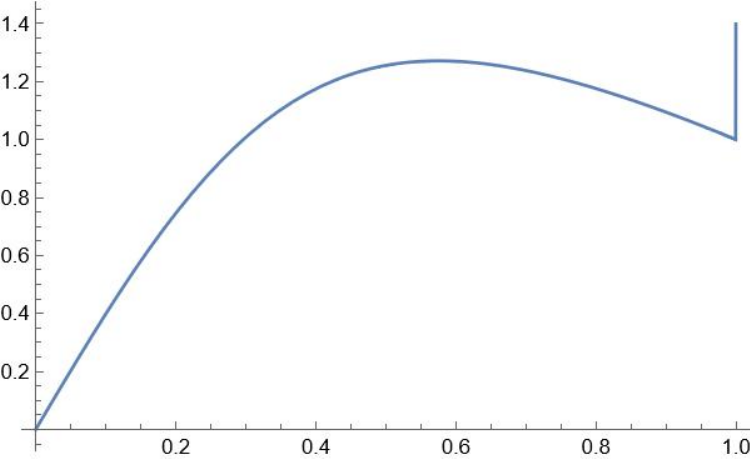}
 	\caption{The  image of $[0,1)$  under the mapping $G_{\infty} $ for the case when  $\alpha=1$ and $n=6$.}
 	\label{fig1}
 \end{figure}
\end{enumerate}

\end{rem}

This paper is organized as follows.
In Section \ref{sec-2}, some  known
results are recalled, and several preliminary results are proved.
In Section \ref{sec-3}, we present the
proof of Theorem \ref{thm-1.1} for the case $p\in(1,\infty)$.
 In Section \ref{sec-4}, we show Theorem 1.1 for the special cases $p=1,2,\infty$.

\section{Preliminaries}\label{sec-2}
In order to prove the main result, we need some preparation.
First, let us recall two known results, which will be used later on.

\begin{Thm}\label{Lem-A}\cite[Theorem 6.12]{ABR}
	If $X$ is a separable normed linear space, then every norm-bounded sequence in $X^{*}$ contains a weak* convergent subsequence.
\end{Thm}

\begin{Thm}\label{Lem-B}\cite[Theorem 2.3]{Liu04}
	Let $u\in C^{2}(\mathbb{B}^{n},\mathbb{R}^{n})$   satisfying $\Delta_{\alpha} u=0$. Then
 \begin{equation}\label{eq-2.1}
	u(x)=\sum_{k=0}^{\infty} \Phi_{k}^{\alpha}\left(|x|^{2}\right)|x|^{k} u_{k}\left(x^{\prime}\right) ,
 \end{equation}
	where  $x=|x|x'\in \mathbb{B}^{n}$, $u_{k} \in \mathcal{H}_{k}(\mathbb{S}^{n-1},\mathbb{R}^{n})$ and 	$ 	\Phi_{k}^{\alpha} $ is the mapping defined in \eqref{eq-1.5}.
The   series in \eqref{eq-2.1} converges uniformly and absolutely on every compact  set  in $\mathbb{B}^{n}$.
\end{Thm}

The following result is  about the maximum principle for the solutions to the Dirichlet problems \eqref{eq-1.2}.

\begin{lem}\label{lem-2.1}
Suppose that   $u\in C^{2} (\mathbb{B}^{n} ,\mathbb{R}^{n})$ satisfies $\Delta_{\alpha} u=0$ in $\mathbb{B}^{n}$
 with $\alpha\in(-\frac{1}{2},0]\cup [\frac{n}{2}-1 ,\infty)$.
For any  open set $\Omega $ with $\overline{\Omega} \subset \mathbb{B}^{n}$,
if there exists   a point $x_{0}\in\Omega $ such that $\max_{x\in \overline{\Omega}} |u(x)|=|u(x_{0})|$,
then $u$ is a constant in $ \Omega $.
\end{lem}
\begin{proof}
For any $x\in\mathbb{B}^{n}$, let $u(x)=\big(u_{1}(x), \ldots,u_{n}(x)\big)$.
For each fixed point $\xi=(\xi_{1}, \ldots,\xi_{n})\in \mathbb{S}^{n-1}$,
  the inner product  $\langle u(x),\xi\rangle= \sum_{i=1}^{n} u_{j}(x)\xi_{j}$
 is a  real-valued   mapping satisfying $\Delta_{\alpha} \langle u(x),\xi\rangle=0$ w.r.t $x\in\mathbb{B}^{n}$.
Let $\Omega $ be an  open set with $\overline{\Omega} \subset \mathbb{B}^{n}$.
  Then \cite[Theorem 3.5]{GT} guarantees that the mapping $x\mapsto\langle u(x),\xi\rangle$ cannot achieve a non-negative maximum in $\Omega$
unless $\langle u(x),\xi\rangle$ is a constant in  $\Omega$.

Suppose that  there exists $x_{0}\in \Omega$ such that
$|u(x_{0})|=\max_{x\in \overline{\Omega}} |u(x)|\not=0$.
Let $\xi_{0}=\frac{u(x_{0})}{|u(x_{0})|}$.
Then for any $x\in \overline{\Omega}$,
$$
 \langle u(x ),\xi_{0}\rangle \leq |u(x)|\leq |u(x_{0})|= \langle u(x_{0}),\xi_{0}\rangle ,
$$
which means that the mapping $x\mapsto  \langle u(x ),\xi_{0}\rangle $ achieves a non-negative maximum at $x_{0}\in \Omega$.
Therefore, $\langle u(x ),\xi_{0}\rangle$ is a constant in $  \Omega $, and so is $u $.
 The proof of the lemma is complete.
\end{proof}

\begin{lem}\label{lem-2.2}
	Let    $u\in C^{2} (\mathbb{B}^{n} ,\mathbb{R}^{n})$ satisfy $\Delta_{\alpha} u=0$ in $\mathbb{B}^{n}$ with $\alpha\in(-\frac{1}{2},0]\cup[\frac{n}{2}-1,+\infty)$.
	 If $p\in (1,\infty]$, then  $u\in\mathcal{H}^{p}(\mathbb{B}^{n}, \mathbb{ R}^{n} )$
	if and only if there exists a mapping $f\in L^{p}\left(\mathbb{S}^{n-1},\mathbb{R}^{n}\right)$ such that $u=P_{\alpha}[f]$.
Further, $\|u\|_{\mathcal{H}^{p}}=\|f\|_{L^{p}}$.

\end{lem}
\begin{proof}
The proof of this lemma is  based upon the ideas from \cite[Proposition 2]{JP99}.
Assume that $u=P_{\alpha}[f]$  with  $f \in L^{p}\left(\mathbb{S}^{n-1},\mathbb{R}^{n}\right)$ and $p\in (1,\infty]$.
For any $r\in[0,1)$ and $\zeta\in \mathbb{S}^{n-1}$,
it follows from \cite[Equation (2.8)]{Liu04} and \cite[Lemma 1.2]{ol2014} that
 \begin{equation}\label{eq-2.2}
\int_{\mathbb{S}^{n-1}}P_{\alpha}(r\zeta,\xi)d\sigma(\xi)=C_{n,\alpha}\,{_{2}F_{1}}\left(-\alpha, \frac{n}{2}-1-\alpha ; \frac{n}{2} ;r^{2}\right)=\frac{\Phi_{0}^{\alpha}\left(r^{2}\right)}{\Phi_{0}^{\alpha}(1)}\leq 1.
 \end{equation}
If  $p\in (1,\infty)$, then  H\"older's inequality ensures that
\begin{align*}
	|u(r \zeta)|^{p} &\leq   \int_{\mathbb{S}^{n-1}} P_{\alpha}(r \zeta, \xi)|f(\xi)|^{p} d \sigma(\xi).
\end{align*}
 This, together with \eqref{eq-2.2} and Fubini's Theorem, guarantees that
\begin{align}\label{eq-2.3}
	\|u\|_{\mathcal{H}^{p}}^{p}
&=\sup _{0<r<1} \int_{\mathbb{S}^{n-1}}|u(r \zeta)|^{p} d \sigma(\zeta)\\\nonumber
	&\leq   \sup _{0<r<1} \int_{\mathbb{S}^{n-1}}\int_{\mathbb{S}^{n-1}} P_{\alpha}(r \zeta, \xi)|f(\xi)|^{p} d \sigma(\xi)d\sigma(\zeta)\\\nonumber
	&=   \sup _{0<r<1} \int_{\mathbb{S}^{n-1}}|f(\xi)|^{p}\int_{\mathbb{S}^{n-1}}P_{\alpha}(r \zeta, \xi)d\sigma(\zeta)d\sigma(\xi)\\\nonumber
	&\leq   \|f\|^{p}_{L^{p}} .
\end{align}
If $p=\infty$, then \eqref{eq-2.2} implies that
\begin{align}\label{eq-2.4}
	\left|u(r\zeta)\right|&\leq\int_{\mathbb{S}^{n-1}}\left|f(\xi)\right|P_{\alpha}(r\zeta,\xi)d\sigma(\xi)
	 \leq\|f\|_{L^{\infty}}\int_{\mathbb{S}^{n-1}}P_{\alpha}(r\zeta,\xi)d\sigma(\xi)
	 \leq \|f\|_{L^{\infty}} .
\end{align}
Based on the above inequalities, we see that $u\in\mathcal{H}^{p}(\mathbb{B}^{n-1},\mathbb{R}^{n})$
for  $p\in(1,\infty]$.

Conversely, if $u\in\mathcal{H}^{p}(\mathbb{B}^{n}, \mathbb{ R}^{n} )$ with $p\in (1,\infty]$, then the  $L^{p}$-norms of the mapping $\zeta \mapsto u(r \zeta)$  are uniformly bounded w.r.t $r\in[0,1)$. It follows from Theorem A that there exists a sequence  $\{r_{j}\}\subset[0,1)$ such that $\lim_{j \rightarrow \infty} r_{j}=1$
  and a mapping  $f \in L^{p}(\mathbb{S}^{n-1},\mathbb{R}^{n})$  such that $u\left(r_{j}\xi\right)$ converges weak* to   $f(\xi)$ in  $ L^{p}\left(\mathbb{S}^{n-1}\right)$ as $j\rightarrow\infty$.
 For any $r\in[0,1)$ and $\zeta\in \mathbb{S}^{n-1}$,
let $u_{r}(\zeta)=u (r\zeta)$.
  Then \cite[Inequality (6.8)]{ABR} ensures that
  \begin{align}\label{eq-2.5}
\|f\|_{L^{p}}\leq \text{liminf}_{j\rightarrow\infty }\|u_{r_{j}} \|_{L^{p}}\leq
\sup_{r\in[0,1)}\|u_{r} \|_{L^{p}}=\|u \|_{\mathcal{H}^{p}}.
\end{align}
On the other hand, since for each $q\in[1,\infty)$ and $r\zeta\in \mathbb{B}^{n}$,
 $$\int_{\mathbb{S}^{n-1}} |P _{\alpha}(r \zeta, \xi)|^{q}  d \sigma(\xi)<\infty,$$
we obtain from the weak* convergence of $\{u(r_{j}\xi)\}$ that
$$
P_{\alpha}[f](r \zeta)  =\lim _{j \rightarrow \infty} \int_{\mathbb{S}^{n-1}} P_{\alpha}(r \zeta, \xi) u\left(r_{j} \xi\right) d \sigma(\xi).
$$
Further, by Theorem B, we see that there exists a sequence $\{u_{k}\} \subseteq  \mathcal{H}_{k}(\mathbb{S}^{n-1},\mathbb{R}^{n})$ such that
$$u(x)=\sum_{k=0}^{\infty} \Phi_{k}^{\alpha}\left(|x|^{2}\right)|x|^{k} u_{k}\left(x^{\prime}\right)
 $$
with  $x=|x|x'\in \mathbb{B}^{n}$,
and the series $\sum_{k=0}^{\infty} P_{\alpha}(r \zeta, \xi) \Phi_{k}^{\alpha}\left( r_{j} ^{2}\right) r_{j} ^{k} u_{k}\left(\xi\right)$
 converges uniformly and absolutely on every compact  set  in $\mathbb{B}^{n}$.
 Hence,
 \begin{align}\label{eq-2.6}
P_{\alpha}[f](r \zeta)&=\lim _{j \rightarrow \infty} \int_{\mathbb{S}^{n-1}}\sum_{k=0}^{\infty} P_{\alpha}(r \zeta, \xi) \Phi_{k}^{\alpha}\left( r_{j} ^{2}\right) r_{j} ^{k} u_{k}\left(\xi\right)d \sigma(\xi) \\  \nonumber
&=\lim _{j \rightarrow \infty}\sum_{k=0}^{\infty} \Phi_{k}^{\alpha}\left( r_{j} ^{2}\right) r_{j} ^{k} \int_{\mathbb{S}^{n-1}} P_{\alpha}(r \zeta, \xi) u_{k}\left(\xi\right)d \sigma(\xi).
 \end{align}
Because $\{u_{k}\} \subseteq  \mathcal{H}_{k}(\mathbb{S}^{n-1},\mathbb{R}^{n})$, we see that $u_{k}\in C(\mathbb{S}^{n-1}, \mathbb{R}^{n})$.
Then it follows from  \cite[Theorem 2.4]{Liu04} that
 \begin{equation}\label{eq-2.7}
\int_{\mathbb{S}^{n-1}} P_{\alpha}(r \zeta, \xi) u_{k}\left(\xi\right)d \sigma(\xi)
=\frac{\Phi_{k}^{\alpha}\left(|r|^{2}\right)}{\Phi_{k}^{\alpha}(1)}r^{k}u_{k} (\zeta).
 \end{equation}
Since  Theorem B ensures that
the series $\sum_{k=0}^{\infty} \Phi_{k}^{\alpha}\left( r_{j} ^{2}\right) r_{j} ^{k}\frac{\Phi_{k}^{\alpha}\left(|r|^{2}\right)}{\Phi_{k}^{\alpha}(1)}r^{k}u_{k} (\zeta)$
 converges uniformly and absolutely on every compact  set  in $\mathbb{B}^{n}$,
 we obtain from \eqref{eq-2.6} and \eqref{eq-2.7} that
  \begin{align}\label{eq-2.8}
P_{\alpha}[f](r \zeta)
 &=\lim _{j \rightarrow \infty}\sum_{k=0}^{\infty} \Phi_{k}^{\alpha}\left( r_{j} ^{2}\right) r_{j} ^{k}\frac{\Phi_{k}^{\alpha}\left(|r|^{2}\right)}{\Phi_{k}^{\alpha}(1)}r^{k}u_{k} (\zeta)\\  \nonumber
&=\sum_{k=0}^{\infty} \lim _{j \rightarrow \infty}
\frac{\Phi_{k}^{\alpha}\left( r_{j} ^{2}\right)}{\Phi_{k}^{\alpha}(1)} r_{j} ^{k}
\Phi_{k}^{\alpha}\left(|r|^{2}\right)r^{k}u_{k} (\zeta)\\  \nonumber
&=\sum_{k=0}^{\infty}  \Phi_{k}^{\alpha}\left(|r|^{2}\right)r^{k}u_{k} (\zeta).
 \end{align}
Substituting  \eqref{eq-2.1} into \eqref{eq-2.8} gives that
 $$
P_{\alpha}[f](r \zeta) =u  (r\zeta).
$$

Further, by \eqref{eq-2.3}$\sim$\eqref{eq-2.5}, we obtain that
$$\|f\|_{L^{p}} =\|u \|_{\mathcal{H}^{p}} $$
for $p\in(1,\infty]$.
The proof of the lemma is complete.
\end{proof}

\begin{lem}\label{lem-2.3}
	For $\alpha\in(-\frac{1}{2},\infty)$, $q \in[1, \infty)$, $r \in[0,1)$ and $a \in \mathbb{R}$,
	\begin{align*}
		\frac{\partial}{\partial a} \int_{\mathbb{S}^{n-1}} \mid P_{\alpha}\left(r e_{n}, \eta\right) & -\left.a\right|^{q} d \sigma(\eta) \\
		= & \int_{\mathbb{S}^{n-1}} q\left(a-P_{\alpha}\left(r e_{n}, \eta\right)\right)\left|P_{\alpha}\left(r e_{n}, \eta\right)-a\right|^{q-2} d \sigma(\eta) .
	\end{align*}	
\end{lem}
\begin{proof}
We consider the case when $q \in[1,2)$ and the case when $q \in   [2, \infty)$, separately.

\begin{case}\label{case-2.1} $q\in[1,2)$. \end{case}

For $(r, a) \in[0,1) \times \mathbb{R}$, because
\begin{align*}
	\int_{\mathbb{S}^{n-1}} \frac{\partial}{\partial a} \mid P_{\alpha}\left(r e_{n}, \eta\right) & -\left.a\right|^{q} d \sigma(\eta) \\
	& =\int_{\mathbb{S}^{n-1}} q\left(a-P_{\alpha}\left(r e_{n}, \eta\right)\right)\left|P_{\alpha}\left(r e_{n}, \eta\right)-a\right|^{q-2} d \sigma(\eta)
\end{align*}
and
$$
\int_{\mathbb{S}^{n-1}}\left|P_{\alpha}\left(r e_{n}, \eta\right)-a\right|^{q-1} d \sigma(\eta) \leq C_{n,\alpha}^{q-1}\left(\frac{1+r}{1-r}\right)^{(1+2\alpha)(q-1)}\frac{1}{(1-r)^{(n-2\alpha-2)(q-1)}}+\left|a\right|^{q-1},
$$
then by \cite[Proposition 2.4] {kp} or \cite{TE01}, we obtain that
\begin{align*}
	\frac{\partial}{\partial a} \int_{\mathbb{S}^{n-1}} \mid P_{\alpha}\left(r e_{n}, \eta\right) & -\left.a\right|^{q} d \sigma(\eta) \\
	= & \int_{\mathbb{S}^{n-1}} q\left(a-P_{\alpha}\left(r e_{n}, \eta\right)\right)\left|P_{\alpha}\left(r e_{n}, \eta\right)-a\right|^{q-2} d \sigma(\eta) .
\end{align*}

 \begin{case}\label{case-2.2} $p\in [2,\infty)$. \end{case}

By direct calculations, we have
$$
\frac{\partial}{\partial a}\left|P_{\alpha}\left(r e_{n}, \eta\right)-a\right|^{q}=q\left(a-P_{\alpha}\left(r e_{n}, \eta\right)\right)\left|P_{\alpha}\left(r e_{n}, \eta\right)-a\right|^{q-2} .
$$
Obviously, the mappings

$$(r, a, \eta) \mapsto\left|P_{\alpha}\left(r e_{n}, \eta\right)-a\right|^{q} \quad \text { and } \quad(r, a, \eta) \mapsto \frac{\partial}{\partial a}\left|P_{\alpha}\left(r e_{n}, \eta\right)-a\right|^{q}
$$
are continuous in $[0,1) \times \mathbb{R} \times \mathbb{S}^{n-1}$. Therefore, for any $(r, a) \in[0,1) \times \mathbb{R}$,
$$
\frac{\partial}{\partial a} \int_{\mathbb{S}^{n-1}}\left|P_{\alpha}\left(r e_{n}, \eta\right)-a\right|^{q} d \sigma(\eta)=\int_{\mathbb{S}^{n-1}} \frac{\partial}{\partial a}\left|P_{\alpha}\left(r e_{n}, \eta\right)-a\right|^{q} d \sigma(\eta),
$$
as required. The proof of the lemma is completed.
\end{proof}
For $\alpha\in(-\frac{1}{2},\infty)$, $q \in(1, \infty)$, $r \in(0,1)$ and $a \in \mathbb{R}$, let
\begin{equation}\label{eq-2.9}
	F(r, a)=\int_{\mathbb{S}^{n-1}}\left(P_{\alpha}\left(r e_{n}, \eta\right)-a\right)\left|P_{\alpha}\left(r e_{n}, \eta\right)-a\right|^{q-2} d \sigma(\eta).
\end{equation}
Then we have the following results on $F(r, a)$.

\begin{lem}\label{lem-2.4}
For $\alpha\in(-\frac{1}{2},\infty)$, $q\in(1,\infty)$, $r\in(0,1)$ and $a\in \mathbb{R}$,
\begin{equation}\label{eq-2.10}
\partial_aF(r,a)=(1-q)\int_{\mathbb{S}^{n-1}} |P_\alpha(re_{n},\eta)-a|^{q-2} d\sigma(\eta)
\end{equation}
and
\begin{equation}\label{eq-2.11}
\partial_rF(r,a)=(q-1)\int_{\mathbb{S}^{n-1}} \partial_{ r} P_\alpha(re_{n},\eta)\cdot |P_\alpha(re_{n},\eta)-a|^{q-2} d\sigma(\eta).
\end{equation}
Furthermore, for any $[\mu_{1},\mu_{2}]\subset (0,1)$ and $[\nu_{1},\nu_{2}]\subset (0,\infty)$,
the above two integrals are
uniformly convergent w.r.t. $(r,a)\in[\mu_{1},\mu_{2}]\times [\nu_{1},\nu_{2}]$.

\end{lem}
\begin{proof}
In order to prove this lemma, we only need to prove \eqref{eq-2.10} and the uniformly convergence of the integral in  \eqref{eq-2.10} since \eqref{eq-2.11}  and the uniformly convergence of the integral in  \eqref{eq-2.11} can be proved in a similar way.
For this, we consider the case when $q\in(1,2)$ and the case when $q\in [2,\infty)$, separately.

\begin{case}\label{case-2.3} $q\in(1,2)$. \end{case}
For fixed $r\in(0,1)$ and $\eta=(\eta_{1},\ldots,\eta_{n})\in\mathbb{S}^{n-1}$,
by calculations, we know that
\begin{align}\label{eq-2.12}
 \left|P_{\alpha}\left(r e_{n}, \eta\right)-a\right|^{q-2}
\leq 4^{\frac{n}{2}+\alpha} \left|C_{n,\alpha}\left(1-r^{2}\right)^{1+2\alpha}-a\left(1+r^{2}-2 r \eta_{n}\right)^{\frac{n}{2}+\alpha}\right|^{q-2}.
\end{align}
Hence, in order to prove the uniformly convergence of
 $\int_{\mathbb{S}^{n-1}}\left|P_{\alpha}\left(r e_{n}, \eta\right)-a\right|^{q-2} d \sigma(\eta) $ w.r.t. $(r,a)\in[\mu_{1},\mu_{2}]\times [\nu_{1},\nu_{2}]$,
we only need to show the uniformly convergence of
 $\int_{\mathbb{S}^{n-1}}\left|C_{n,\alpha}\left(1-r^{2}\right)^{1+2\alpha}-a\left(1+r^{2}-2 r \eta_{n}\right)^{\frac{n}{2}+\alpha}\right|^{q-2}d \sigma(\eta) $  w.r.t. $(r,a)\in[\mu_{1},\mu_{2}]\times [\nu_{1},\nu_{2}]$,
where $[\mu_{1},\mu_{2}]\subset (0,1)$ and $[\nu_{1},\nu_{2}]\subset (0,\infty)$.

If $a \leq 0$, then
\begin{align}\label{eq-2.13}
 \int_{\mathbb{S}^{n-1}}&\left(C_{n,\alpha}\left(1-r^{2}\right)^{1+2\alpha}+|a|\left(1+r^{2}-2 r \eta_{n}\right)^{\frac{n}{2}+\alpha}\right)^{q-2} d \sigma(\eta)\notag \\
	&\leq \left(|C_{n,\alpha}|\left(1-r^{2}\right)^{1+2\alpha}+|a|(1-r)^{n+2\alpha}\right)^{q-2}.
\end{align}
This, together with \eqref{eq-2.12}, guarantees  that the integral  $ \int_{\mathbb{S}^{n-1}}\left|P_{\alpha}\left(r e_{n}, \eta\right)-a\right|^{q-2} d \sigma(\eta) $
is uniformly convergent w.r.t. $(r,a)\in[\mu_{1},\mu_{2}]\times [\nu_{1},\nu_{2}]\subset (0,1)\times(0,\infty)$.

If $a>0$, by using a spherical coordinate transformation (cf. \cite[Section 2.2]{chen2018}), we obtain
\begin{align}\label{eq-2.14}
\int_{\mathbb{S}^{n-1}} &\left|C_{n,\alpha}\left(1-r^{2}\right)^{1+2\alpha}-a\left(1+r^{2}-2 r \eta_{n}\right)^{\frac{n}{2}+\alpha}\right|^{q-2} d \sigma(\eta)\\ \nonumber
	=& \int_{0}^{\frac{\pi}{2}} \sin ^{n-2} \theta\left|C_{n,\alpha}\left(1-r^{2}\right)^{1+2\alpha}-a\left(1+r^{2}-2 r \cos \theta\right)^{\frac{n}{2}+\alpha}\right|^{q-2} d \theta\\ \nonumber
&+ \int_{\frac{\pi}{2}}^{\pi} \sin ^{n-2} \theta\left|C_{n,\alpha}\left(1-r^{2}\right)^{1+2\alpha}-a\left(1+r^{2}-2 r \cos \theta\right)^{\frac{n}{2}+\alpha}\right|^{q-2} d \theta \\ \nonumber
=&\int_{0}^{1}\left(1-x^{2}\right)^{\frac{n-3}{2}} |\Psi_{1}(r,a,x)|^{q-2} d x+\int_{0}^{1}\left(1-x^{2}\right)^{\frac{n-3}{2}} |\Psi_{2}(r,a,x)|^{q-2} d x,
\end{align}
 where
\begin{equation}\label{eq-2.15}
\Psi_{1}(r,a,x)= C_{n,\alpha}\left(1-r^{2}\right)^{1+2\alpha}-a\left(1+r^{2}-2 r x\right)^{\frac{n}{2}+\alpha}
\end{equation}
and
\begin{equation}\label{eq-2.16}
\Psi_{2}(r,a,x)= C_{n,\alpha}\left(1-r^{2}\right)^{1+2\alpha}-a\left(1+r^{2}+2 r x\right)^{\frac{n}{2}+\alpha} .
\end{equation}
For the last  two integrals  in \eqref{eq-2.14}, we have the following two  claims.

\begin{claim}\label{claim-2.1} The integral $\int_{0}^{1}\left(1-x^{2}\right)^{\frac{n-3}{2}} |\Psi_{1}(r,a,x)|^{q-2} d x$ is  uniformly convergent w.r.t. $(r,a)\in[\mu_{1},\mu_{2}]\times [\nu_{1},\nu_{2}]\subset (0,1)\times(0,\infty)$.
\end{claim}
For  $(r,a)\in[\mu_{1},\mu_{2}]\times [\nu_{1},\nu_{2}]$, let $\lambda_{1}\in \mathbb{R}$ satisfying
\begin{equation}\label{eq-2.17}
\Psi_{1}(r,a,\lambda_{1})=0.
\end{equation}
By computation, we see that
\begin{equation}\label{eq-2.18}
	\lambda_{1}=\frac{1+r^{2}}{2 r}-\frac{(1-r^{2})^{\frac{2+4\alpha}{n+2\alpha}}}{2 r} a^{-\frac{2}{n+2\alpha}}C_{n,\alpha}^{\frac{2}{n+2\alpha}} .
\end{equation}

If $\lambda_{1} \in[0,1]$, in order to estimate $ \Psi_{1} $, we let $s(t)=t^{\frac{n}{2}+\alpha}$ with $t\in (0,\infty)$. Obviously,
 \begin{center}
 $s'(t)=(\frac{n}{2}+\alpha)t^{\frac{n}{2}+\alpha-1}>0$ and $s''(t)=(\frac{n}{2}+\alpha)(\frac{n}{2}+\alpha-1)t^{\frac{n}{2}+\alpha-2}>0$
\end{center}
for $n\geq3$ and $\alpha>-\frac{1}{2}$, which implies that $s'(t)$ is an increasing mapping.
Hence, for any $t_{1}>t_{2}>0$,
\begin{equation}\label{eq-2.19}
	\left|s(t_{1})-s(t_{2})\right|\geq s'(t_{2})(t_{1}-t_{2}).
\end{equation}

It follows from \eqref{eq-2.15}, \eqref{eq-2.17} and \eqref{eq-2.19} that
\begin{align}\label{eq-2.20}
|\Psi_{1}(a,r,x)|=&|\Psi_{1}(a,r,x)-\Psi_{1}(a,r,\lambda_{1})|\\ \nonumber
=&a\left| \left(1+r^{2}-2 r x\right)^{\frac{n}{2}+\alpha}- \left(1+r^{2}-2 r \lambda_{1}\right)^{\frac{n}{2}+\alpha} \right|\\ \nonumber
\geq&a r(n+2\alpha)\left(1+r^{2}-2 r \lambda_{1}\right)^{\frac{n}{2}+\alpha-1}( \lambda_{1}-x)
\end{align}
 when $0\leq x\leq\lambda_{1}$. Similarly,
\begin{align}\label{eq-2.21}
|\Psi_{1}(a,r,x)|
\geq a r(n+2\alpha)\left(1+r^{2}-2 r x\right)^{\frac{n}{2}+\alpha-1}(x- \lambda_{1})
\end{align}
 when $ 1\geq x\geq\lambda_{1}$.

For any  $(x,r,a)\in[0,1]\times[\mu_{1},\mu_{2}]\times [\nu_{1},\nu_{2}]$,
it follows from \eqref{eq-2.20} and \eqref{eq-2.21} that
\begin{align}\label{eq-2.22}
|\Psi_{1}(a,r,x)|
\geq a r(n+2\alpha)(1-r)^{n+2\alpha-2}|\lambda_{1}-x|
\geq C_{1}|\lambda_{1}-x|,
\end{align}
where
\begin{align}\label{eq-2.22c}
C_{1}=\mu_{1}\nu_{1}(n+2\alpha)(1-\mu_{2})^{n+2\alpha-2}
\end{align}
 and   $\lambda_{1} \in[0,1]$.

Now, for any $\delta>0$, \eqref{eq-2.22} ensures that
\begin{align*}
	\int_{\lambda_{1}-\delta}^{\lambda_{1}}& \left(1-x^{2}\right)^{\frac{n-3}{2}} |\Psi_{1}(r,a,x)|^{q-2}d x
\leq   C_{1}^{q-2} \int_{\lambda_{1}-\delta}^{\lambda_{1}}\left|x-\lambda_{1}\right|^{q-2} d x
=  C_{1}^{q-2}\delta^{q-1} /(q-1),
\end{align*}
and similarly,
\begin{align*}
	\int_{\lambda_{1}}^{\lambda_{1}+\delta}& \left(1-x^{2}\right)^{\frac{n-3}{2}} |\Psi_{1}(r,a,x)|^{q-2}d x
\leq   C_{1}^{q-2}\delta^{q-1} /(q-1).
\end{align*}
The above  inequalities show that $\int_{0}^{1}\left(1-x^{2}\right)^{\frac{n-3}{2}} |\Psi_{1}(r,a,x)|^{q-2} d x$
is  uniformly convergent w.r.t. $(r,a)\in[\mu_{1},\mu_{2}]\times [\nu_{1},\nu_{2}] $ when $\lambda_{1} \in[0,1]$.

Obviously, if $\lambda_{1}>1$ or $\lambda_{1}<0$, then the mapping  $(x,r,a)\mapsto\left(1-x^{2}\right)^{\frac{n-3}{2}} |\Psi_{1}(r,a,x)|^{q-2} $  is continuous in
$[0,1]\times[\mu_{1},\mu_{2}]\times [\nu_{1},\nu_{2}]$, and so $\int_{0}^{1}\left(1-x^{2}\right)^{\frac{n-3}{2}} |\Psi_{1}(r,a,x)|^{q-2} d x$ is  uniformly convergent w.r.t. $(r,a)\in[\mu_{1},\mu_{2}]\times [\nu_{1},\nu_{2}] $.
 Hence,  Claim \ref{claim-2.1} is proved.

\begin{claim}\label{claim-2.2} The integral $\int_{0}^{1}\left(1-x^{2}\right)^{\frac{n-3}{2}} |\Psi_{2}(r,a,x)|^{q-2} d x$ is  uniformly convergent w.r.t. $(r,a)\in[\mu_{1},\mu_{2}]\times [\nu_{1},\nu_{2}]\subset (0,1)\times(0,\infty)$.
\end{claim}
For  $(r,a)\in[\mu_{1},\mu_{2}]\times [\nu_{1},\nu_{2}]$, let $\lambda_{2}\in \mathbb{R}$ satisfying
\begin{align}\label{eq-2.23}
\Psi_{2}(r,a,\lambda_{2})=0.
\end{align}
Then
\begin{align*}
	\lambda_{2}= \frac{(1-r^{2})^{\frac{2+4\alpha}{n+2\alpha}}}{2 r} a^{-\frac{2}{n+2\alpha}}C_{n,\alpha}^{\frac{2}{n+2\alpha}} -\frac{1+r^{2}}{2 r}.
\end{align*}

If $\lambda_{2}\in[0,1]$, then  \eqref{eq-2.16}, \eqref{eq-2.19} and \eqref{eq-2.23} guarantee that
\begin{align*}
	|\Psi_{2}(a,r,x)|=&|\Psi_{2}(a,r,x)-\Psi_{2}(a,r,\lambda_{2})|\\
	=&a\left| \left(1+r^{2}+2 r x\right)^{\frac{n}{2}+\alpha}- \left(1+r^{2}+2 r \lambda_{2}\right)^{\frac{n}{2}+\alpha}\right|\\
	\geq&a r(n+2\alpha)\left(1+r^{2}+2 r x\right)^{\frac{n}{2}+\alpha-1}( \lambda_{2}-x)
\end{align*}
when $0\leq x\leq\lambda_{2}$, and
\begin{align*}
	|\Psi_{2}(a,r,x)|
	\geq a r(n+2\alpha)\left(1+r^{2}+2 r \lambda_{2}\right)^{\frac{n}{2}+\alpha-1}(x- \lambda_{2})
\end{align*}
when $ 1\geq x\geq\lambda_{2}$.

Therefore, for any  $(x,r,a)\in[0,1]\times[\mu_{1},\mu_{2}]\times [\nu_{1},\nu_{2}]$,
\begin{equation}\label{eq-2.24}
	|\Psi_{2}(a,r,x)|
	\geq a r(n+2\alpha)(1-r)^{n+2\alpha-2}|\lambda_{2}-x|
	\geq C_{1}|\lambda_{2}-x|,
\end{equation}
where $C_{1} $ is the constant given by \eqref{eq-2.22c}  and $\lambda_{2}\in[0,1]$.

Now, for any $\delta>0$, \eqref{eq-2.24} ensures that
\begin{align*}
	\int_{\lambda_{2}-\delta}^{\lambda_{2}}& \left(1-x^{2}\right)^{\frac{n-3}{2}} |\Psi_{2}(r,a,x)|^{q-2}d x
	\leq   C_{1}^{q-2} \int_{\lambda_{2}-\delta}^{\lambda_{2}}\left|x-\lambda_{2}\right|^{q-2} d x
	=  C_{1}^{q-2}\delta^{q-1} /(q-1),
\end{align*}
and similarly,
\begin{align*}
	\int_{\lambda_{2}}^{\lambda_{2}+\delta}& \left(1-x^{2}\right)^{\frac{n-3}{2}} |\Psi_{2}(r,a,x)|^{q-2}d x
	\leq   C_{1}^{q-2}\delta^{q-1} /(q-1).
\end{align*}
The above   inequalities show that $\int_{0}^{1}\left(1-x^{2}\right)^{\frac{n-3}{2}} |\Psi_{2}(r,a,x)|^{q-2} d x$
is  uniformly convergent w.r.t. $(r,a)\in[\mu_{1},\mu_{2}]\times [\nu_{1},\nu_{2}] $ when $\lambda_{2} \in[0,1]$.

Obviously, if $\lambda_{2}>1$ or $\lambda_{2}<0$, then the mapping $$(x,r,a)\mapsto\left(1-x^{2}\right)^{\frac{n-3}{2}} |\Psi_{2}(r,a,x)|^{q-2} $$ is continuous in
$[0,1]\times[\mu_{1},\mu_{2}]\times [\nu_{1},\nu_{2}]$, and so $\int_{0}^{1}\left(1-x^{2}\right)^{\frac{n-3}{2}} |\Psi_{2}(r,a,x)|^{q-2} d x$ is  uniformly convergent w.r.t. $(r,a)\in[\mu_{1},\mu_{2}]\times [\nu_{1},\nu_{2}] $.
Hence,  Claim \ref{claim-2.2} is proved.

\medskip

In the following, we will prove  that \eqref{eq-2.10} holds true.
For any $(r,a)\in[\mu_{1},\mu_{2}]\times [\nu_{1},\nu_{2}]\subset (0,1)\times(0,\infty)$,
 by  \eqref{eq-2.22}, we get that
\begin{align}\label{eq-2.25}
	\int_{0}^{1} \left(1-x^{2}\right)^{\frac{n-3}{2}}|\Psi_{1}(r,a,x)|^{q-2} d x
    \leq& C^{q-2}_{1} \int_{0}^{1}\frac{1}{|x-\lambda_{1}|^{2-q}}dx \\ \nonumber
    =&C^{q-2}_{1}
    \frac{\lambda_{1}^{q-1}+(1-\lambda_{1})^{q-1}}{q-1}
\end{align}
when $\lambda_{1}\in[0,1]$. Similarly, by \eqref{eq-2.24}, we obtain that
\begin{align}\label{eq-2.26}
	\int_{0}^{1} \left(1-x^{2}\right)^{\frac{n-3}{2}}|\Psi_{2}(r,a,x)|^{q-2} d x
	\leq&C^{q-2}_{1} \int_{0}^{1}\frac{1}{|x-\lambda_{2}|^{2-q}}dx \\ \nonumber
	=&C^{q-2}_{1} \frac{\lambda_{2}^{q-1}+(1-\lambda_{2})^{q-1}}{q-1}
\end{align}
when $\lambda_{2}\in[0,1]$.

If $\lambda_{1}>1$ or $\lambda_{1}<0$, then the continuity and boundedness of the mapping
$ (x,r,a)\mapsto\left(1-x^{2}\right)^{\frac{n-3}{2}} |\Psi_{1}(r,a,x)|^{q-2}$
in $ [0,1]\times[\mu_{1},\mu_{2}]\times [\nu_{1},\nu_{2}]$
follow from \eqref{eq-2.17}.
Similarly, if $\lambda_{2}>1$ or $\lambda_{2}<0$, then the continuity and boundedness of the mapping
$ (x,r,a)\mapsto\left(1-x^{2}\right)^{\frac{n-3}{2}} |\Psi_{2}(r,a,x)|^{q-2}$
in $ [0,1]\times[\mu_{1},\mu_{2}]\times [\nu_{1},\nu_{2}]$
follow from \eqref{eq-2.23}.

By (\ref{eq-2.25}), (\ref{eq-2.26}) and above discussion, we know that
$$
\int_{0}^{1} \left(1-x^{2}\right)^{\frac{n-3}{2}}\left(|\Psi_{1}(r,a,x)|^{q-2} +|\Psi_{2}(r,a,x)|^{q-2}\right)d x<\infty 
$$
for any $(r,a)\in[\mu_{1},\mu_{2}]\times [\nu_{1},\nu_{2}]\subset (0,1)\times(0,\infty)$.
From this and \eqref{eq-2.12}$\sim$\eqref{eq-2.14},
we see that for any $(r,a)\in[\mu_{1},\mu_{2}]\times [\nu_{1},\nu_{2}]\subset (0,1)\times \mathbb{R}$,
$$\int_{\mathbb{S}^{n-1}} |P_\alpha(re_{n},\eta)-a|^{q-2} d\sigma(\eta)<\infty.$$
This, together with  \cite[Proposition 2.4] {kp} or \cite{TE01}, we see that (\ref{eq-2.10}) is true.

 \begin{case}\label{case-2.4} $p\in [2,\infty)$. \end{case}

By direct calculations, we have
$$
\frac{\partial}{\partial a}\left(P_{\alpha}\left(r e_{n}, \eta\right)-a\right)\left|P_{\alpha}\left(r e_{n}, \eta\right)-a\right|^{q-2}=(1-q)\left|P_{\alpha}\left(r e_{n}, \eta\right)-a\right|^{q-2}.
$$
Obviously, the mappings
$$
(r, a, \eta) \mapsto\left(P_{\alpha}\left(r e_{n}, \eta\right)-a\right)\left|P_{\alpha}\left(r e_{n}, \eta\right)-a\right|^{q-2}
$$
and
$$
(r, a, \eta) \mapsto(1-q)\left|P_{\alpha}\left(r e_{n}, \eta\right)-a\right|^{q-2}
$$
are continuous in $(0,1) \times \mathbb{R} \times \mathbb{S}^{n-1}$. Therefore, for any $(r, a) \in(0,1) \times \mathbb{R}$, \eqref{eq-2.10} is true and $\partial_{a} F(r, a)$ is uniformly convergent w.r.t. $(r, a) \in\left[\mu_{1}, \mu_{2}\right] \times\left[v_{1}, v_{2}\right]$ for any $\left[\mu_{1}, \mu_{2}\right] \subset(0,1)$ and $\left[v_{1}, v_{2}\right] \subset(0, \infty)$. The proof of the lemma is completed.
\end{proof}
\begin{lem}\label{lem-2.5}
For  $\alpha\in(-\frac{1}{2},\infty)$  and $q\in(1,\infty)$, both $\partial_aF(r,a)$ and $\partial_rF(r,a)$ are continuous
w.r.t. $(r,a)\in(0,1)\times(0,\infty)$.
\end{lem}
\begin{proof}
In order to prove this lemma, we only need to prove the continuity of $\partial_aF(r,a)$ since the continuity of $\partial_rF(r,a)$ can be proved in a similar way.
For this, we consider the case when $q\in(1,2)$ and the case when $q\in [2,\infty)$, separately.

\begin{case}\label{case-2.5} $q\in(1,2)$. \end{case}


By using a spherical coordinate transformation, we obtain from  \eqref{eq-2.10}  that
\begin{align*}
-\partial_{a}  F(r, a)
&=(q-1) \int_{\mathbb{S}^{n-1}}\left|P_{\alpha}\left(r e_{n}, \eta\right)-a\right|^{q-2} d \sigma(\eta)
 =\int_{0}^{1} \big(J_{1}(x,r, a)+J_{2}(x,r, a)\big)dx,
\end{align*}
where
\begin{align}\label{eq-2.27}
	J_{1}(x,r, a)   =  (q-1)\left(1-x^{2}\right)^{\frac{n-3}{2}}  \left(1+r^{2}-2 r x\right)^{(\frac{n}{2}+\alpha)(2-q)} \left|\Psi_{1}(r,a,x)\right|^{q-2}
\end{align}
and
\begin{align*}
	J_{2}(x,r, a)   =  (q-1)\left(1-x^{2}\right)^{\frac{n-3}{2}}  \left(1+r^{2}+2 r x\right)^{(\frac{n}{2}+\alpha)(2-q)}  \left|\Psi_{2}(r,a,x)\right|^{q-2} .
\end{align*}
Here $\Psi_{1}$ and $\Psi_{2}$ are  the mappings defined in \eqref{eq-2.15} and \eqref{eq-2.16}, respectively.

In order to prove the continuity of  $\partial_{a} F(r, a)$, we only need to prove the continuity of the mapping
 $(r,a)\mapsto\int_{0}^{1}J_{1}(x,r, a)dx$ since the continuity of  the mapping
 $(r,a)\mapsto\int_{0}^{1}J_{2}(x,r, a)dx$ can be proved in a similar way.

Let $\lambda_{1}=\lambda_{1} (r, a )$ be given by \eqref{eq-2.18},
where $\lambda_{1}=\lambda_{1}(r,a)$ means that $\lambda_{1}$ depends only on $r$ and $a$.
Obviously, if  $\lambda_{1}>1$ or $\lambda_{1}<0$, then \eqref{eq-2.27} and the continuity  of
$  \Psi_{1} ^{q-2}$ w.r.t.  $(x,r,a)\in [0,1]\times[\mu_{1},\mu_{2}]\times [\nu_{1},\nu_{2}]$ guarantee
the continuity of $\int_{0}^{1}J_{1}(x,r, a)dx$ w.r.t. $(r,a)\in [\mu_{1},\mu_{2}]\times [\nu_{1},\nu_{2}]$ for any $[\mu_{1},\mu_{2}]\subset (0,1)$ and $[\nu_{1},\nu_{2}]\subset (0,\infty)$.
 Therefore, $\int_{0}^{1}J_{1}(x,r, a)dx$ is continuous
w.r.t. $(r,a)\in(0,1)\times(0,\infty)$.

If $\lambda_{1}\in[0,1]$,  we only need to prove that $\int_{0}^{1}J_{1}(x,r, a)dx$ is continuous at every fixed point
$(r_{0},a_{0})\in(0,1)\times(0,\infty)$.
Assume that
$(r_{0},a_{0})\in(\mu_{1},\mu_{2})\times (\nu_{1},\nu_{2})
\subset (0,1)\times(0,\infty)$
and $(r_{0}+\Delta r,a_{0}+\Delta a)
\in(\mu_{1},\mu_{2})\times (\nu_{1},\nu_{2})$.

By Claim \ref{claim-2.1} and  \eqref{eq-2.27}, we know that $\int_{0}^{1}J_{1}(x,r, a)dx$ is uniformly convergent w.r.t. $(r,a)\in \left[\mu_{1}, \mu_{2}\right] \times   \left[\nu_{1}, \nu_{2}\right]$  for any  $\left[\mu_{1}, \mu_{2}\right] \subset(0,1)$  and  $\left[\nu_{1}, \nu_{2}\right] \subset(0, \infty)$. Without loss of generality, we assume that  $\lambda_{1}\left(r_{0}, a_{0}\right) \in(0,1)$. Then for any $\varepsilon_{1}>0$, there exist constants  $\iota_{1}=\iota_{1}\left(\varepsilon_{1}\right) \rightarrow 0^{+}$ and  $\iota_{2}=\iota_{2}\left(\varepsilon_{1}\right) \rightarrow 0^{+}$ such that for any $(r, a) \in\left[\mu_{1}, \mu_{2}\right] \times\left[\nu_{1}, \nu_{2}\right] $,
$$
\left|\int_{\lambda_{1}\left(r_{0}, a_{0}\right)-\iota_{1}}^{\lambda_{1}\left(r_{0}, a_{0}\right)+\iota_{2}}
 J_{1}(x,r, a)d x\right|<\varepsilon_{1} .
$$
Since
\begin{align*}
  \left|\int_{0}^{1} \right.& J_{1} (x,r_{0}+ \Delta r, a_{0}+\Delta a )dx - \left. \int_{0}^{1}  J_{1} (x,r_{0}, a_{0} ) dx \right| \\ \nonumber
=&\left| \int_{0}^{\lambda_{1} (r_{0}, a_{0} )-\iota_{1}} \big(J_{1} (x,r_{0}+ \Delta r, a_{0}+\Delta a )- J_{1} (x,r_{0}, a_{0} ) \big)\right. d x \\ \nonumber
	& +\int_{\lambda_{1} (r_{0}, a_{0} )-\iota_{1}}^{\lambda_{1} (r_{0}, a_{0} )+\iota_{2}}\big(J_{1} (x,r_{0}+ \Delta r, a_{0}+\Delta a )- J_{1} (x,r_{0}, a_{0} ) \big) d x \\ \nonumber
	& \left.+\int_{\lambda_{1} (r_{0}, a_{0} )+\iota_{2}}^{1}\big(J_{1} (x,r_{0}+ \Delta r, a_{0}+\Delta a )- J_{1} (x,r_{0}, a_{0} ) \big) d x \right|,
\end{align*}
we obtain that
\begin{align}\label{eq-2.28}
  \left|\int_{0}^{1} \right.& J_{1} (x,r_{0}+ \Delta r, a_{0}+\Delta a )dx - \left. \int_{0}^{1}  J_{1} (x,r_{0}, a_{0} ) dx \right| \\ \nonumber
\leq&\;\; \left|  \int_{0}^{\lambda_{1} (r_{0}, a_{0} )-\iota_{1}}\big(J_{1} (x,r_{0}+ \Delta r, a_{0}+\Delta a )- J_{1} (x,r_{0}, a_{0} ) \big) \right.d x \\ \nonumber
	&\left. +\int_{\lambda_{1} (r_{0}, a_{0} )+\iota_{2}}^{1}\big(J_{1} (x,r_{0}+ \Delta r, a_{0}+\Delta a )- J_{1} (x,r_{0}, a_{0} ) \big) d x  \right| +2\varepsilon_{1} .
\end{align}

By \eqref{eq-2.18}, it is easy to see that $\lambda_{1}(r, a)$  is uniformly continuous in $\left[\mu_{1}, \mu_{2}\right] \times   \left[\nu_{1}, \nu_{2}\right]$. Then for any $\iota^{\prime} \in\left(0, \min \left\{\iota_{1}, \iota_{2}\right\}\right)$, there exists a constant  $\iota_{3}=\iota_{3}\left(\iota^{\prime}\right) \rightarrow   0^{+}$ such that for any $(r, a) \in\left[r_{0}-\iota_{3}, r_{0}+\iota_{3}\right] \times\left[a_{0}-\iota_{3}, a_{0}+\iota_{3}\right] \subset\left[\mu_{1}, \mu_{2}\right] \times   \left[\nu_{1}, \nu_{2}\right]
$,
$$
\lambda_{1}(r, a) \in\left(\lambda_{1}\left(r_{0}, a_{0}\right)-\frac{1}{2} \iota^{\prime}, \lambda_{1}\left(r_{0}, a_{0}\right)+\frac{1}{2} \iota^{\prime}\right) \subset\left(\lambda_{1}\left(r_{0}, a_{0}\right)-\iota_{1}, \lambda_{1}\left(r_{0}, a_{0}\right)+\iota_{2}\right) .
$$
This, together with \eqref{eq-2.17} and \eqref{eq-2.27}, implies that the mapping  $  J_{1}(r, a,x)$  is continuous (also uniformly continuous) in
$$
\left[r_{0}-\iota_{3}, r_{0}+\iota_{3}\right] \times\left[a_{0}-\iota_{3}, a_{0}+\iota_{3}\right] \times\left[0, \lambda_{1}\left(r_{0}, a_{0}\right)-\iota_{1}\right]
$$
and
$$
\left[r_{0}-\iota_{3}, r_{0}+\iota_{3}\right] \times\left[a_{0}-\iota_{3}, a_{0}+\iota_{3}\right] \times\left[\lambda_{1}\left(r_{0}, a_{0}\right)+\iota_{2}, 1\right],
$$
respectively. Therefore, there exists $\iota_{4}=\iota_{4}\left(\varepsilon_{1}\right) \leq \iota_{3}$  such that for all $|\Delta r|<\iota_{4}$,  $|\Delta a|<\iota_{4}$ and for all $x \in\left[0, \lambda_{1}\left(r_{0}, a_{0}\right)-\iota_{1}\right] \cup\left[\lambda_{1}\left(r_{0}, a_{0}\right)+\iota_{2}, 1\right]$,
\begin{equation}\label{eq-2.29}
	\left|J_{1}(r_{0}+\Delta r, a_{0}+\Delta a, x)-J_{1}(r_{0}, a_{0} ,x)\right|<\varepsilon_{1} .
\end{equation}
Then by \eqref{eq-2.28} and \eqref{eq-2.29}, we see that
$$
 \left|\int_{0}^{1}  J_{1} (x,r_{0}+ \Delta r, a_{0}+\Delta a )dx -   \int_{0}^{1}  J_{1} (x,r_{0}, a_{0} ) dx \right| \leq 3 \varepsilon_{1},
$$
which means that the mapping $(r,a)\mapsto\int_{0}^{1}J_{1}(x,r, a)dx$  is continuous at  $\left(r_{0}, a_{0}\right)$.

\begin{case}\label{case-2.6} $q\in[2,\infty)$. \end{case}

Obviously, the mapping
$$
(r, a, \eta) \mapsto(1-q)\left|P_{\alpha}\left(r e_{n}, \eta\right)-a\right|^{q-2}
$$
is continuous in $(0,1) \times(0, \infty) \times \mathbb{S}^{n-1}$. By \eqref{eq-2.10}, we know that  $\partial_{a} F(r, a)$ is continuous w.r.t. $(r, a) \in(0,1) \times(0, \infty)$.
\end{proof}

For $q\in[1,\infty)$, $r\in[0,1)$ and $a\in \mathbb{R}$, define
\begin{equation}\label{eq-2.30}
 \Phi_{q }(r,a)=\left(\int_{\mathbb{S}^{n-1}} |P_\alpha(re_{n},\eta)-a|^qd\sigma(\eta)\right)^{1/q}.
\end{equation}
 Then by Lemmas \ref{lem-2.3}$\sim$\ref{lem-2.5}, we obtain the following result for $\Phi_{q }(r,a)$.
\begin{lem}\label{lem-2.6}
For $\alpha\in(-\frac{1}{2},\infty)$, $q\in(1,\infty)$ and $r\in[0,1)$,
there exists a unique constant $a^{*}=a(r)\in(0,\infty)$ such that $$\Phi_{q }(r,a^*)=\min_{a\in \mathbb{R}}\Phi_{q }(r,a),$$
where $a(0)=C_{n,\alpha}$ and
$\frac{d}{dr}a(r)$ is continuous in $(0,1)$.
\end{lem}

\begin{proof}
When $r=0$, it follows from \eqref{eq-1.4} and \eqref{eq-2.30} that $a^{*}=C_{n,\alpha}$.
Hence, to prove the lemma, it remains to consider the case when $r\in(0,1)$.

For $r\in(0,1)$, $t\in(0,1)$ and $b,c\in\mathbb{R}$ with $b\not=c$, by the Minkowski inequality, we obtain
$$\Phi_{q}\big(r,t b+(1-t)c\big)
<t \Phi_{q} (r,b)+(1-t) \Phi_{q} (r,c),$$
which means that $\Phi_{q} (r,a)$ is strictly convex w.r.t. $a\in\mathbb{R}$.
Furthermore, by Lemma \ref{lem-2.3}, we know that for $q\in(1,\infty)$, $r\in(0,1)$ and $a\in \mathbb{R}$,
\begin{equation}\label{eq-2.31}
\frac{\partial}{\partial a}\Phi_{q} (r,a)
=-\left(\int_{\mathbb{S}^{n-1}} |P_\alpha(re_{n},\eta)-a|^{q}d\sigma(\eta)\right)^{1/q-1}F(r,a),
\end{equation}
where $F(r,a)$ is the mapping from \eqref{eq-2.9}.
Therefore,
$$
\frac{\partial}{\partial a}\Phi_{q}(r,0)=-\int_{\mathbb{S}^{n-1}} |P_\alpha(re_{n},\eta)|^{q-1}d\sigma(\eta)
\left(\int_{\mathbb{S}^{n-1}}  |P_\alpha(re_{n},\eta)|^{q}d\sigma(\eta)\right)^{1/q-1}<0.
$$
These, together with the fact $\lim_{a\rightarrow\infty}\Phi_{q}(r,a)=\infty$,
show that for any  fix $r\in(0,1),$
the mapping $\Phi_{q,r}(a)=:\Phi_{q}(r,a)$ has only one stationary point in $(0,\infty)$ which is its minimum, i.e., $a^{*}=a(r)$.

By \eqref{eq-2.31}, we see that for any $r\in(0,1)$,
$\frac{\partial}{ \partial a}\Phi_{q }\big(r,a(r)\big)=0$ is equivalent to
\begin{equation}\label{eq-2.32}
F\big(r,a(r)\big)=0.
\end{equation}
Furthermore, Lemmas \ref{lem-2.4} and \ref{lem-2.5} tell us that both $\partial_aF(r,a)$ and $\partial_rF(r,a)$ are continuous w.r.t. $(r,a)\in(0,1)\times (0,\infty)$ and that $\partial_aF(r,a)<0$.
Therefore, it follows from \eqref{eq-2.32} and the implicit function theorem
that the derivative of $a^*=a(r)$ is continuous in $(0,1)$ and
$$
\frac{d a(r)}{d r}
=-\frac{ \partial_{r} F(r,a) }
{ \partial_{a} F(r,a) }.
$$
The proof of the lemma is completed.
\end{proof}

\section{Proof of the main result}\label{sec-3}
The aim of this section is to prove the part of Theorem \ref{thm-1.1} when $p\in(1,\infty)$.
In fact, it can be derived directly from Lemmas \ref{lem-3.1}$\sim$\ref{lem-3.4}.
First, we obtain the following estimate on $|u|$.
 \begin{lem}\label{lem-3.1}
For $\alpha\in(-\frac{1}{2},\infty)$  and  $p\in(1,\infty)$, suppose that $u=P_{\alpha}[f]$ and $u(0)=0$, where
$f\in L^p(\mathbb{S}^{n-1},\mathbb{R}^{n})$.
Then
\begin{equation}\label{eq-3.1}
|u(x)|\le G_p(|x|)\|f\|_{L^{p}}
\end{equation}
 in $\mathbb{B}^n$,
 where $G_p (r)$ is the mapping given by \eqref{eq-1.7} and  $\frac{d}{dr}G_p (r)$ is continuous in $(0,1)$ with $G_p(0)=0$.
The inequality  is sharp.
 \end{lem}

\begin{proof} Let $p\in(1,\infty)$ and let $q$ be its conjugate.
For any $x\in \mathbb{B}^{n}$ and $a\in \mathbb{R}^{n}$, it follows from the assumption $u=P_{\alpha}[f]$ and $u(0)=0$ that
$$
u(x)=\int_{\mathbb{S}^{n-1}} (P_{\alpha}(x,\eta)-a)f(\eta)d\sigma(\eta).
$$

If $x=re_{n}$ for some $r\in [0,1)$, then by using H\"older's inequality, we have
\begin{equation}\label{eq-3.2}
	|u(re_{n})| \leqslant\left(\int_{\mathbb{S}^{n-1}}\left|P_{\alpha}\left(r e_{n}, \eta\right)-a\right|^{q} d \sigma(\eta)\right)^{1 / q}\|f\|_{L^{p}}=\Phi_{q }(r,a)\|f\|_{L^{p}},
\end{equation}
where $\Phi_{q }(r,a)$ is the mapping from \eqref{eq-2.30}.

 If for any $r\in[0,1)$, $x\ne re_{n}$, then we choose a unitary transformation $A$ such that $A(re_{n})=x$.
Since for any $w\in \mathbb{B}^{n}$,
 \begin{align*}
P_{\alpha}[f](A(w))
&=C_{n,\alpha}\int_{\mathbb{S}^{n-1}}\frac{(1-|A(w)|^{2})^{1+2\alpha}}{\left|A(w)-\eta\right|^{n+2\alpha}}f(\eta)d\sigma(\eta)\\
&=C_{n,\alpha}\int_{\mathbb{S}^{n-1}}\frac{(1-|w|^{2})^{1+2\alpha}}{\left|w-A^{-1}(\eta)\right|^{n+2\alpha}}f(\eta)d\sigma(\eta)\\
&=C_{n,\alpha}\int_{\mathbb{S}^{n-1}}\frac{(1-|w|^{2})^{1+2\alpha}}{\left|w- \xi \right|^{n+2\alpha}}f(A(\xi))d\sigma( \xi)\\
 &=P_{\alpha}[f\circ A](w),
  \end{align*}
  we see that
  $$
  u(x)=P_{\alpha}[f](A(re_{n}))=P_{\alpha}[f\circ A](re_{n}).
  $$
Notice that $\|f\circ A\|_{L^{p}}=\|f\|_{L^{p}}$.
By   replacing $u$ with  $P_{\alpha}[f\circ A]$ and replacing $f$ with $f\circ A$ in \eqref{eq-3.2}, respectively,
we see   that   for any $x\in \mathbb{B}^{n}$,
\begin{equation}\label{eq-3.3}
|u(x)|=\big|P_{\alpha}[f\circ A](re_{n})\big|\leq \Phi_{q }(r,a)\|f\circ A\|_{L^{p}}= \Phi_{q }(r,a)\|f\|_{L^{p}}.
\end{equation}

Further, for $r\in[0,1)$ and $q\in(1,\infty)$,
by \eqref{eq-1.8}, \eqref{eq-2.30} and Lemma \ref{lem-2.6}, we get
$$\min_{a\in \mathbb{R}}\Phi_{q }(r,a)
=\Phi_{q }(r,a^*)
=\Phi_{q }\big(r,a(r)\big)=G_p(r).
$$
This, together with \eqref{eq-3.3},
implies that \eqref{eq-3.1} holds true.
Further, by \eqref{eq-2.30} and Lemma \ref{lem-2.6}, we know   that $G_p(0)=0$.

Next, we show the smoothness of  $G_p(r)$. Since $\Phi_{q }\big(r,a(r)\big)=G_p(r)$, we see that
$$
\frac{d G_{p}(r)}{dr}
=\left.\frac{\partial \Phi(r, a )}{\partial r} \right|_{a=a(r)}+\left.\frac{\partial\Phi(r, a )}{\partial a} \right|_{a=a(r)} \frac{d a(r)}{dr}
.
$$
Then \eqref{eq-2.31} and \eqref{eq-2.32} ensure that
$$
\frac{d G_{p}(r)}{dr}
=\left.\frac{\partial \Phi(r, a )}{\partial r} \right|_{a=a(r)}
.
$$
It follows from \eqref{eq-2.30} that
\begin{align*}
 \frac{\partial \Phi(r, a )}{\partial r}
= &\int_{\mathbb{S}^{n-1}} \partial_{ r} P_\alpha(re_{n},\eta)  \cdot \left(P_\alpha(re_{n},\eta)-a\right)\cdot |P_\alpha(re_{n},\eta)-a|^{q-2} d\sigma(\eta)\\
	& \times \left(\int_{\mathbb{S}^{n-1}}\left|P_\alpha(re_{n},\eta)-a\right|^{q} d \sigma(\eta)\right)^{1 / q-1}.
\end{align*}
Obviously, the mapping $(r,a)\mapsto \int_{\mathbb{S}^{n-1}}\left|P_\alpha(re_{n},\eta)-a\right|^{q} d \sigma(\eta)$
is positive and continuous in $(0,1)\times(0,\infty)$.
The continuity of the mapping
$$(r,a)\mapsto \int_{\mathbb{S}^{n-1}} \partial_{ r} P_\alpha(re_{n},\eta)  \cdot \left(P_\alpha(re_{n},\eta)-a\right)\cdot |P_\alpha(re_{n},\eta)-a|^{q-2} d\sigma(\eta)$$
in $ (0,1)\times\mathbb{R}$ easily follows from the continuity of $\partial_rF(r,a)$ which is given by \eqref{eq-2.11},  and we omit it.
By these and the continuity of  $\frac{d}{dr}a(r)$, we see that
   $\frac{d}{dr}G_p (r)$ is continuous in $(0,1)$.

 The proof of the sharpness of inequality \eqref{eq-3.1} is similar to that of \cite[Inequality (2.17)]{Chen21},
 where $P_{\alpha}$ is used
instead of $P_{h}$ and $f_{x}$ is used
instead of $\phi_{x}$, and we omit it.
\end{proof}

\begin{lem}\label{lem-3.2}
For $\alpha\in(-\frac{1}{2},\infty)$  and  $p\in(1,\infty]$, suppose that $u=P_{\alpha}[f]$ and $u(0)=0$, where
$f\in L^p(\mathbb{S}^{n-1},\mathbb{R}^{n})$.
Then
\begin{equation}\label{eq-3.4}
	\|D u(0)\| \leq C_{n,\alpha}(n+2\alpha)\left(\frac{\Gamma\left(\frac{n}{2}\right) \Gamma\left(\frac{1+q}{2}\right)}{\sqrt{\pi} \Gamma\left(\frac{n+q}{2}\right)}\right)^{\frac{1}{q}}\|f\|_{L^{p}}.
\end{equation}
The inequality is sharp.
 \end{lem}

\begin{proof}
Let  $f=\left(f_{1}, \ldots, f_{n}\right)$  and $u=\left(u_{1}, \ldots, u_{n}\right)$.
 For $r \in[0,1)$ and $i \in\{1, \ldots, n\}$, \cite[Proposition 2.4] {kp} guarantees that
$$\nabla u_{i} (r e_{n} )=\int_{\mathbb{S}^{n-1}} \nabla P_{\alpha} (r e_{n}, \eta ) f_{i}(\eta) d \sigma(\eta),$$
where the gradients $\nabla u_{i}$ are understood as column vectors and
$$
\nabla P_{\alpha}\left(r e_{n}, \eta\right)=C_{n,\alpha}\frac{\left(1-r^{2}\right)^{2\alpha}\left((n+2\alpha)\left(1-r^{2}\right)\left(\eta-r e_{n}\right)-2(1+2\alpha)r e_{n}\left|\eta-r e_{n}\right|^{2}\right)}{\left|\eta-r e_{n}\right|^{n+2+2\alpha}}.$$

Then for any $\xi \in \mathbb{S}^{n-1}$, we have
\begin{align*}
	D u(0) \xi & =\left(\left\langle\nabla u_{1}(0), \xi\right\rangle, \ldots,\left\langle\nabla u_{n}(0), \xi\right\rangle\right)^{T} \\
	& =C_{n,\alpha}(n+2\alpha) \int_{\mathbb{S}^{n-1}}\langle\eta, \xi\rangle f(\eta) d \sigma(\eta),
\end{align*}
where $T$ is the transpose. Since
\begin{align*}
	\max _{\xi \in \mathbb{S}^{n-1}}\left|\int_{\mathbb{S}^{n-1}}\langle\eta, \xi\rangle f(\eta) d \sigma(\eta)\right| 
& \leq \max _{\xi \in \mathbb{S}^{n-1}}\left(\int_{\mathbb{S}^{n-1}}|\langle\eta, \xi\rangle|^{q} d \sigma(\eta)\right)^{\frac{1}{q}}\|f\|_{L^{p}} \\
	& =\left(\int_{\mathbb{S}^{n-1}}\left|\eta_{n}\right|^{q} d \sigma(\eta)\right)^{\frac{1}{q}}\|f\|_{L^{p}}
\end{align*}
and
\begin{equation*}
	\int_{\mathbb{S}^{n-1}}\left|\eta_{n}\right|^{q} d \sigma(\eta)=\frac{\Gamma\left(\frac{n}{2}\right) \Gamma\left(\frac{1+q}{2}\right)}{\sqrt{\pi} \Gamma\left(\frac{n+q}{2}\right)}:=\alpha_{q},
\end{equation*}
we obtain
$$
\|D u(0)\|=\max _{\xi \in \mathbb{S}^{n-1}}|D u(0) \xi| \leq C_{n,\alpha}(n+2\alpha) \alpha_{q}^{\frac{1}{q}}\|f\|_{L^{p}},
$$
where $\eta=\left(\eta_{1}, \ldots, \eta_{n}\right) \in \mathbb{S}^{n-1}$.

To prove the sharpness of inequality \eqref{eq-3.4}, for any $i \in\{1, \ldots, n\}$, let
$$
f(\eta)=\left|\eta_{i}\right|^{\frac{q}{p}} \cdot \operatorname{sign}\left(\eta_{i}\right).
$$
Then the similar reasoning as in \cite[p.632]{Chen21}   shows the sharpness of \eqref{eq-3.4}.
\end{proof}

The following two results are some properties of $G_p$ which is given by \eqref{eq-1.7}.
 \begin{lem}\label{lem-3.3}
	For $\alpha\in (-\frac{1}{2},0]\cup [\frac{n}{2}-1,\infty)$ and $p \in(1, \infty)$, $G_{p}:[0,1) \rightarrow[0, \infty)$ is an strictly increasing mapping with $G_{p}(0)=0$.
 \end{lem}
\begin{proof}

  The proof of this result is similar  to that of \cite[Lemma 2.7]{Chen21},
 where $P_{\alpha}$ is used
instead of $P_{h}$,
Lemma \ref{lem-2.2} is used
instead of  \cite[Equality (1.1)]{Chen21},
  Lemma \ref{lem-2.1} is used
instead of \cite[Theorem 4.4.2(a)]{sto2016}, and we omit it.
\end{proof}

\begin{lem}\label{lem-3.4}
For $p \in(1, \infty)$ and  $r \in[0,1)$, $G_{p}(r)$ is derivable at $r=0$ and
$$
G_{p}^{\prime}(0)=C_{n,\alpha}(n+2\alpha)\left(\frac{\Gamma\left(\frac{n}{2}\right) \Gamma\left(\frac{1+q}{2}\right)}{\sqrt{\pi} \Gamma\left(\frac{n+q}{2}\right)}\right)^{\frac{1}{q}}.
$$	
 \end{lem}
\begin{proof}
  The proof of this result is similar  to that of \cite[Lemma 2.8]{Chen21},
 where $P_{\alpha}$ is used
instead of $P_{h}$,
  the constant $2(n-1)\alpha_{q}^{\frac{1}{q}}$ in the proof of  \cite[Lemma 2.8]{Chen21}
is replaced by $C_{n,\alpha}(n+2\alpha)\alpha_{q}^{\frac{1}{q}}$, and we omit it.
\end{proof}

%
%
%
 \section{Special cases}\label{sec-4}

\subsection{The case $p=\infty$}
For any $x \in \mathbb{B}^{n}$, let $A$ be a unitary transformation such that $A\left(r e_{n}\right)=x$, where $r=|x|$. Since $u(0)=0$  and $u=P_{\alpha}[f]$, we have
$$
u(x)=u\left(A\left(r e_{n}\right)\right)=\int_{\mathbb{S}^{n-1}}\left(P_{\alpha}\left(A\left(r e_{n}\right), \eta\right)-C_{n,\alpha}\frac{\left(1-r^{2}\right)^{1+2\alpha}}{\left(1+r^{2}\right)^{\frac{n}{2}+\alpha}}\right) f(\eta) d \sigma(\eta) .
$$
For any $\eta \in \mathbb{S}^{n-1}$, let  $\xi=A^{-1} \eta$. Then
\begin{align}\label{eq-4.1}
	|u(x)|\leq&\|f\|_{L^{\infty}} \int_{\mathbb{S}^{n-1}}\left|P_{\alpha}\left(r e_{n}, \xi\right)-C_{n,\alpha}\frac{\left(1-r^{2}\right)^{1+2\alpha}}{\left(1+r^{2}\right)^{\frac{n}{2}+\alpha}}\right| d \sigma(\xi)  \\ \nonumber
=&\|f\|_{L^{\infty}} \int_{\mathbb{S}_{+}^{n-1}}\left(P_{\alpha}\left(r e_{n}, \xi\right)-C_{n,\alpha}\frac{\left(1-r^{2}\right)^{1+2\alpha}}{\left(1+r^{2}\right)^{\frac{n}{2}+\alpha}}\right) d \sigma(\xi) \\ \nonumber
	&+\|f\|_{L^{\infty}} \int_{\mathbb{S}_{-}^{n-1}}\left(C_{n,\alpha}\frac{\left(1-r^{2}\right)^{1+2\alpha}}{\left(1+r^{2}\right)^{\frac{n}{2}+\alpha}}-P_{\alpha}\left(r e_{n}, \xi\right)\right) d \sigma(\xi)  \\ \nonumber
=&U_{\alpha}\left(r e_{n}\right)\|f\|_{L^{\infty}},
\end{align}
where $U_{\alpha}$ is the mapping in Theorem \ref{thm-1.1}. By letting
$$
f(\eta)=C \operatorname{sign}\left(P_{\alpha}\left(r e_{n}, A^{-1} \eta\right)-C_{n,\alpha}\frac{\left(1-r^{2}\right)^{1+2\alpha}}{\left(1+r^{2}\right)^{\frac{n}{2}+\alpha}}\right)
$$
on $\mathbb{S}^{n-1}$, we obtain the sharpness of \eqref{eq-4.1}, where $C$ is a constant.

Further, by Lemma \ref{lem-3.2}, we obtain that the inequality \eqref{eq-1.9} holds for $p=\infty$ and this inequality is also sharp.

Next, we discuss the property of $G_{\infty}$. It follows from \eqref{eq-4.1} that
\begin{equation}\label{eq-4.2}
	G_{\infty}(r)=\int_{\mathbb{S}^{n-1}}\left|P_{\alpha}\left(r e_{n}, \eta\right)-a^{*}\right| d \sigma(\eta)=U_{\alpha}\left(r e_{n}\right),
\end{equation}
where  $a^{*}=C_{n,\alpha}\frac{\left(1-r^{2}\right)^{1+2\alpha}}{\left(1+r^{2}\right)^{n+2\alpha}}$.
Obviously, $G_{\infty}(0)=0$.

In the following, we compute the values of $U_{\alpha}\left(r e_{n}\right)$  (or  $G_{\infty}(r)$), where  $r \in[0,1)$. By using a spherical coordinate transformation, we obtain that
$$
U_{\alpha}\left(r e_{n}\right)
=C_{n,\alpha}\left(1-r^{2}\right)^{1+2\alpha} \frac{\Gamma\left(\frac{n}{2}\right)}{\sqrt{\pi} \Gamma\left(\frac{n-1}{2}\right)} \int_{0}^{\pi} \frac{\sin ^{n-2} \theta}{\left(1+r^{2}-2 r \cos \theta\right)^{\frac{n}{2}+\alpha}}\left(\chi_{\mathbb{S}_{+}^{n-1}}-\chi_{\mathbb{S}_{-}^{n-1}}\right) d \theta .
$$
Elementary calculations lead to
\begin{align*}
	&\int_{0}^{\pi} \frac{\sin ^{n-2} \theta}{\left(1+r^{2}-2 r \cos \theta\right)^{\frac{n}{2}+\alpha}}\left(\chi_{\mathbb{S}_{+}^{n-1}}-\chi_{\mathbb{S}_{-}^{n-1}}\right) d \theta \\
	&=\int_{0}^{\frac{\pi}{2}}\left(\frac{\sin ^{n-2} \theta}{\left(1+r^{2}-2 r \cos \theta\right)^{\frac{n}{2}+\alpha}}-\frac{\sin ^{n-2} \theta}{\left(1+r^{2}+2 r \cos \theta\right)^{\frac{n}{2}+\alpha}}\right) d \theta\\
	&=\frac{1}{\left(1+r^{2}\right)^{\frac{n}{2}+\alpha}} \sum_{k=0}^{\infty}\left(\int_{0}^{\pi / 2} \cos ^{k} \theta \sin ^{n-2} \theta d \theta\right)\frac{\Gamma\left(k+\alpha+\frac{n}{2}\right)}{\Gamma\left(\alpha+\frac{n}{2}\right)\Gamma\left(k+1\right)}\left(1-(-1)^{k}\right)\left(\frac{2r}{1+r^{2}}\right)^{k}.
\end{align*}
Since
$$
\int_{0}^{\pi / 2} \cos ^{k} \theta \sin ^{n-2} \theta d \theta=\frac{\Gamma\left(\frac{1+k}{2}\right) \Gamma\left(\frac{n-1}{2}\right)}{2 \Gamma\left(\frac{k+n}{2}\right)}
$$
(cf. \cite[p.19]{rain}), then
we obtain
\begin{align*}
U(re_{n})\;=&\;\frac{\Gamma\left(\frac{n}{2}+\alpha\right) \Gamma(1+\alpha)}{\Gamma\left(\frac{n}{2}\right) \Gamma(1+2 \alpha)}\frac{\Gamma\left(\frac{n}{2}\right)}{\sqrt{\pi} \Gamma\left(\frac{n-1}{2}\right)} \frac{\left(1-r^{2}\right)^{1+2\alpha}}{\left(1+r^{2}\right)^{\frac{n}{2}+\alpha}} \\
&\;\times\sum_{k=0}^{\infty} \frac{\Gamma\left(\frac{1+k}{2}\right) \Gamma\left(\frac{n-1}{2}\right)\Gamma\left(k+\alpha+\frac{n}{2}\right)}{2 \Gamma\left(\frac{k+n}{2}\right)\Gamma\left(\alpha+\frac{n}{2}\right)\Gamma\left(k+1\right)}\left(1-(-1)^{k}\right) \left(\frac{2r}{1+r^{2}}\right)^{k}\\
=&\;\frac{\Gamma(1+\alpha)}{\sqrt{\pi}\Gamma(1+2\alpha)}\frac{2r(1-r^{2})^{1+2\alpha}}{(1+r^{2})^{1+\alpha+\frac{n}{2}}}\\
&\;\times\sum_{k=0}^{\infty}\frac{\Gamma(k+1)\Gamma(2k+1+\alpha+\frac{n}{2})}{\Gamma(k+\frac{n+1}{2})\Gamma(2k+2)}\left(\frac{2r}{1+r^{2}}\right)^{2k}.
\end{align*}
By Legebdre's duplication formula (cf. \cite[p.24]{rain}), we get
\begin{align*}
	\frac{\Gamma(2 k+1+\alpha+\frac{n}{2})}{\Gamma\left(2k+2\right)}
	=\frac{2^{\alpha+\frac{n}{2}-1} \Gamma\left(k+\frac{n+2+2\alpha}{4}\right) \Gamma\left(k+\frac{n+4+2\alpha}{4}\right)}{ \Gamma\left(k+\frac{3}{2}\right) \Gamma(k+1)} .
\end{align*}
Therefore
\begin{align*}
	U_{\alpha}(re_{n})=&\;\frac{2^{\alpha+\frac{n}{2}}\Gamma(1+\alpha)}{\sqrt{\pi}\Gamma(1+2\alpha)}\frac{r(1-r^{2})^{1+2\alpha}}{(1+r^{2})^{1+\alpha+\frac{n}{2}}}\\
	&\;\times\sum_{k=0}^{\infty}\frac{\Gamma(k+1)\Gamma\left(k+\frac{n+2+2\alpha}{4}\right) \Gamma\left(k+\frac{n+4+2\alpha}{4}\right)}{\Gamma(k+\frac{n+1}{2}) \Gamma\left(k+\frac{3}{2}\right)\Gamma(k+1) }\left(\frac{4r^{2}}{(1+r^{2})^{2}}\right)^{k}\\
	=&\;\frac{2^{\alpha+\frac{n}{2}}\Gamma(1+\alpha)\Gamma(\frac{n+2+2\alpha}{4})\Gamma(\frac{n+4+2\alpha}{4})}{\sqrt{\pi}\Gamma(1+2\alpha)\Gamma(\frac{n+1}{2})\Gamma(\frac{3}{2})}\frac{r(1-r^{2})^{1+2\alpha}}{(1+r^{2})^{1+\alpha+\frac{n}{2}}}\\
	&\;\times\sum_{k=0}^{\infty}\frac{
	(1)_{k}(\frac{n+2+2\alpha}{4})_{k}(\frac{n+4+2\alpha}{4})_{k}}{(\frac{n+1}{2})_{k}(\frac{3}{2})_{k}k!}\left(\frac{4r^{2}}{(1+r^{2})^{2}}\right)^{k}\\
	=&\;\frac{2\Gamma(1+\alpha)\Gamma(1+\alpha+\frac{n}{2})r\left(1-r^{2}\right)^{1+2\alpha}}{\sqrt{\pi}\Gamma(1+2\alpha)\Gamma(\frac{1+n}{2})\left(1+r^{2}\right)^{1+\alpha+\frac{n}{2}}}\times\\
	&\;\;\;{_{3}F_{2
	}}\left(1,\frac{n+2+2\alpha}{4},\frac{n+4+2\alpha}{4};\frac{3}{2},\frac{n+1}{2};\frac{4r^{2}}{(1+r^{2})^{2}}\right).
\end{align*}

Obviously,  $G_{\infty}(r)=U_{\alpha}\left(r e_{n}\right)$ is differentiable in $[0,1)$.
Moreover, by replacing  $f_{*}=\chi_{\mathbb{S}_{+}^{n-1}-} \chi_{\mathbb{S}_{-}^{n-1}}$  and  $u_{*}=U_{\alpha}$, respectively,
the similar reasoning as in the proof of \cite[Claim 2.3]{Chen21} shows that $G_{\infty} $ is an increasing homeomorphism in $[0,1)$.

Since \eqref{eq-2.2}, \eqref{eq-4.1}, \eqref{eq-4.2} and Lemma \ref{lem-2.2} yield that
$$
1 \geq \sup _{r \in[0,1)} U_{\alpha}\left(r e_{n}\right)=\sup _{r \in[0,1)} G_{\infty}(r) \geq \frac{\left\|U_{\alpha}\right\|_{\mathcal{H}^{\infty}}}{\left\|\chi_{\mathbb{S}_{+}^{n-1}}-\chi_{\mathbb{S}_{-}^{n-1}}\right\|_{L^{\infty}}}=1,
$$
we see that $\lim _{r \rightarrow 1^{-}} G_{\infty}(r)=1$. Therefore, $ G_{\infty}$  maps  $[0,1)$  onto  $[0,1)$.
\medskip

\subsection{The case $p=2$}
In this case we deal with the extremal problem
$$G_{2}(r)=\left(\inf _{a \in \mathbb{R}} \int_{\mathbb{S}^{n-1}}\left|P_{\alpha}\left(r e_{n}, \eta\right)-a\right|^{2} d \sigma(\eta)\right)^{1 / 2}.
$$
By \cite[Lemma 2.1]{Liu04}, we obtain that
\begin{align*}
	&\int_{\mathbb{S}^{n-1}}  \left|P_{\alpha}\left(r e_{n}, \eta\right)-a\right|^{2} d \sigma(\eta) \\
=& \int_{\mathbb{S}^{n-1}} P_{\alpha}^{2}(re_{n}, \eta) d \sigma(\eta)+a^{2} \int_{\mathbb{S}^{n-1}} d \sigma(\eta)-2 a \int_{\mathbb{S}^{n-1}} P_{\alpha}\left(r e_{n}, \eta\right) d \sigma(\eta) \\
	=&\;\;C_{n,\alpha}^{2}\left(1-r^{2}\right)^{2+4\alpha}{_{2}F_{1}}\left(n+2\alpha, \frac{n}{2}+2\alpha+1 ; \frac{n}{2} ; r^{2}\right)+a^{2}\\
	&\;\;-2aC_{n,\alpha}(1-r^{2})^{1+2\alpha}{_{2}F_{1}}\left(\frac{n}{2}+\alpha,\alpha+1 ; \frac{n}{2} ; r^{2}\right).
\end{align*}
So $a^{*}=\;C_{n,\alpha}(1-r^{2})^{1+2\alpha}{_{2}F_{1}}\left(\frac{n}{2}+\alpha,\alpha+1 ; \frac{n}{2} ; r^{2}\right)$ and
\begin{align*}
	&\left(\inf _{a\in\mathbb{R}} \int_{\mathbb{S}^{n-1}} \left| P_{\alpha}\left(r e_{n}, \eta\right)-a\right|^{2} d \sigma(\eta)\right)^{1 / 2}
    =C_{n,\alpha}(1-r^{2})^{1+2\alpha}\\
	&\times \sqrt{{_{2}F_{1}}\left(n+2\alpha, \frac{n}{2}+2\alpha+1 ; \frac{n}{2} ; r^{2}\right)-
\left({_{2}F_{1}}\left(\frac{n}{2}+\alpha,\alpha+1 ; \frac{n}{2} ; r^{2}\right)\right)^{2}}.
\end{align*}
This, together with \eqref{eq-3.1}, implies that
\begin{align*}
	&\frac{|u(x)|}{C_{n,\alpha}  (1-r^{2})^{1+2\alpha}}\\
	&\leq  \|f\|_{L^{2}}
 \sqrt{{_{2}F_{1}}\left(n+2\alpha, \frac{n}{2}+2\alpha+1 ; \frac{n}{2} ; r^{2}\right)-\left({_{2}F_{1}}\left(\frac{n}{2}+\alpha,\alpha+1 ; \frac{n}{2} ; r^{2}\right)\right)^{2}}.
\end{align*}
\subsection{The case $p=1$}
In this case we have
$$
G_{1}(r)=\inf _{a \in \mathbb{R}} \sup _{\eta \in \mathbb{S}^{n-1}}\left|P_{\alpha}\left(r e_{n}, \eta\right)-a\right|
$$
for any $r \in[0,1)$. Since
$$
\max _{\eta \in \mathbb{S}^{n-1}} P_{\alpha}\left(r e_{n}, \eta\right)=C_{n,\alpha}\frac{(1-r^{2})^{1+2\alpha}}{(1-r)^{n+2\alpha}} \text { and } \min _{\eta\in \mathbb{S}^{n-1}} P_{\alpha}\left(r e_{n}, \eta\right)=C_{n,\alpha}\frac{(1-r^{2})^{1+2\alpha}}{(1+r)^{n+2\alpha}},
$$
we easily conclude that
$$
a^{*}=\frac{1}{2}C_{n,\alpha}\left(\frac{(1-r^{2})^{1+2\alpha}}{(1-r)^{n+2\alpha}}+\frac{(1-r^{2})^{1+2\alpha}}{(1+r)^{n+2\alpha}}\right)
$$
and
$$
G_{1}(r)=\frac{1}{2}C_{n,\alpha}\left(\frac{(1-r^{2})^{1+2\alpha}}{(1-r)^{n+2\alpha}}-\frac{(1-r^{2})^{1+2\alpha}}{(1+r)^{n+2\alpha}}\right).
$$
Obviously, $G_{1}(r)$ is a strictly increasing mapping from $[0,1)$  onto $[0, \infty)$ and $\frac{d}{dr}G_1 (r)$ is continuous in $[0,1)$.
Then for any  $x \in \mathbb{B}^{n} $ with  $|x|=r $, we have
\begin{equation}\label{eq-4.3}
	|u(x)| \leq G_{1}(r)\|f\|_{L^{1}}=\frac{1}{2}C_{n,\alpha}\left(\frac{(1-r^{2})^{1+2\alpha}}{(1-r)^{n+2\alpha}}-\frac{(1-r^{2})^{1+2\alpha}}{(1+r)^{n+2\alpha}}\right)\|f\|_{L^{1}}.
\end{equation}
This, together with the fact  $u(x)=D u(0) x+o(x)$, implies
\begin{equation}\label{eq-4.4}
	\|D u(0)\| \leq \limsup _{x \rightarrow 0} \frac{|u(x)|}{|x|} \leq \limsup _{r \rightarrow 0^{+}} \frac{G_{1}(r)}{r}\|f\|_{L^{1}}=C_{n,\alpha}(n+2\alpha)\|f\|_{L^{1}}.
\end{equation}
Here, $o(x)$  is a vector satisfying $\lim _{|x| \rightarrow 0^{+}} \frac{|o(x)|}{|x|}=0$.

Now, we show the sharpness of \eqref{eq-4.3} and \eqref{eq-4.4}. For  $i \in \mathbb{Z}^{+}, \eta \in \mathbb{S}^{n-1}$  and  $x \in \mathbb{B}^{n}$, we let
$$
f_{i}(\eta)=\frac{\chi_{\Omega_{i}}(\eta)}{2\left\|\chi_{\Omega_{i}}\right\|_{L^{1}}}-\frac{\chi_{\Omega_{i}^{\prime}}(\eta)}{2\left\|\chi_{\Omega_{i}^{\prime}}\right\|_{L^{1}}} \quad \text { and } \quad u_{i}(x)=P_{\alpha}\left[f_{i}\right](x) \text {, }
$$
where $\Omega_{i}=\left\{\eta \in \mathbb{S}^{n-1}:\left|\eta-e_{n}\right| \leq \frac{1}{i}\right\}$  and  $\Omega_{i}^{\prime}=\left\{\eta \in \mathbb{S}^{n-1}:\left|\eta+e_{n}\right| \leq \frac{1}{i}\right\}$.
Then the similar reasoning as in \cite[p.640 and 641]{Chen21}   shows the sharpness of \eqref{eq-4.3} and \eqref{eq-4.4}.

\vspace*{5mm}
\noindent{\bf Funding.}
 The first author was partly supported by NNSF of China (No. 12071121),
NSF of Hunan Province (No. 2022JJ30365), SRF of Hunan Provincial Education Department (No. 22B0034)
 and the construct program of the key discipline in Hunan Province.

\end{document}